\documentclass[a4paper]{article}
\usepackage[latin1]{inputenc}
\usepackage[english]{babel}

\usepackage{amsmath}
\usepackage{amsthm}
\usepackage{amssymb}  
\usepackage{color}
\usepackage{bbm}
\usepackage{mathtools} 
\usepackage{wasysym} 
 
\usepackage{tikz,pgfplots,subcaption}
  
\usetikzlibrary{plotmarks,calc}
\theoremstyle{plain} 
\newtheorem{theorem}{Theorem}
\newtheorem{proposition}[theorem]{Proposition}
\newtheorem{lemma}[theorem]{Lemma}
\newtheorem{corollary}[theorem]{Corollary}
\newtheorem{remark}[theorem]{Remark}

\newtheorem{example}[theorem]{Example} 
                      
\newcommand{\R}{\mathbb{R}}

\newcommand{\mL}{\mathcal{L}}
\newcommand{\mM}{\mathcal{M}}

\newcommand{\dd}{\, \text{d}}

\newcommand{\Ab}{A_{\mathrm{b}}}
\newcommand{\Ac}{A_{\mathrm{c}}}
\newcommand{\Acr}{A_{\mathrm{c}r}}
\newcommand{\Ahatc}{\widehat{A}_{\mathrm{c}}}
\newcommand{\Atildec}{\widetilde{A}_{\mathrm{c}}}
\newcommand{\Bb}{B_\mathrm{b}}
\newcommand{\BbUpper}{B_{\mathrm{b}1}}
\newcommand{\Bc}{B_{\mathrm{c}}}
\newcommand{\Bcr}{B_{\mathrm{c}r}}
\newcommand{\Bhatc}{\widehat{B}_{\mathrm{c}}}
\newcommand{\Btildec}{\widetilde{B}_{\mathrm{c}}}
\newcommand{\Cb}{C_{\mathrm{b}}}
\newcommand{\Cc}{C_{\mathrm{c}}}
\newcommand{\Ccr}{C_{\mathrm{c}r}}
\newcommand{\Chatc}{\widehat{C}_{\mathrm{c}}}
\newcommand{\Ctildec}{\widetilde{C}_{\mathrm{c}}}
\newcommand{\Etildec}{\widetilde{E}_{\mathrm{c}}}
\newcommand{\Fk}{F_{\mathrm{k}}}
\newcommand{\Fs}{F_{\mathrm{s}}}
\newcommand{\Hc}{\mathcal{H}_\mathrm{c}}
\newcommand{\Jb}{J_\mathrm{b}}
\newcommand{\Jhatc}{\widehat{J}_{\mathrm{c}}}
\newcommand{\Lc}{L_{\mathrm{c}}}
\newcommand{\Lcalc}{\mL_{\mathrm{c}}}

\newcommand{\Mcalo}{\mM_{\mathrm{o}}}
\newcommand{\Lf}{L_{\mathrm{f}}}
\newcommand{\nc}{n_{\mathrm{c}}}
\newcommand{\pAbs}{\tilde{p}}
\newcommand{\Pc}{\mathcal{P}_\mathrm{c}}
\newcommand{\Phatc}{\widehat{\mathcal{P}}_{\mathrm{c}}}
\newcommand{\Phatf}{\widehat{\mathcal{P}}_{\mathrm{f}}}
\newcommand{\Pf}{\mathcal{P}_\mathrm{f}}
\newcommand{\Ptildec}{\widetilde{\mathcal{P}}_{\mathrm{c}}}
\newcommand{\Ptildef}{\widetilde{\mathcal{P}}_{\mathrm{f}}}
\newcommand{\Qb}{Q_\mathrm{b}}
\newcommand{\QbUpperLeft}{Q_{\mathrm{b}11}}
\newcommand{\QbUpperRight}{Q_{\mathrm{b}12}}
\newcommand{\QbLowerLeft}{Q_{\mathrm{b}21}}
\newcommand{\QbLowerRight}{Q_{\mathrm{b}22}}
\newcommand{\Qf}{\mathcal{Q}_\mathrm{f}}
\newcommand{\Qhatc}{\widehat{Q}_{\mathrm{c}}}
\newcommand{\Rb}{R_\mathrm{b}}
\newcommand{\Rc}{R_\mathrm{c}}
\newcommand{\Rf}{\mathcal{R}_\mathrm{f}}
\newcommand{\Rhatc}{\widehat{R}_{\mathrm{c}}}
\newcommand{\tEnd}{t_{\mathrm{end}}}
\newcommand{\uc}{u_{\mathrm{c}}}
\newcommand{\Vb}{V_\mathrm{b}}
\newcommand{\Wb}{W_\mathrm{b}}
\newcommand{\xb}{x_{\mathrm{b}}}
\newcommand{\xc}{x_{\mathrm{c}}}
\newcommand{\xdotc}{\dot{x}_{\mathrm{c}}}
\newcommand{\yc}{y_{\mathrm{c}}}

\newcommand{\bmat}[1]{\begin{bmatrix}#1\end{bmatrix}}

\usepackage{hyperref}
\usepackage{vmargin}
\setmarginsrb{3cm}{2.5cm}{3cm}{1.5cm}{0cm}{0cm}{0cm}{1.5cm} 

\newcommand{\matlab}{MATLAB\textsuperscript{\textregistered}}

\newlength\fheight
\newlength\fwidth
\usepackage{algorithmic}
\usepackage{algorithm}

\begin{document}

\title{Error bounds for port-Hamiltonian model and controller reduction based on system balancing}

\maketitle
 
  \centerline{\scshape Tobias Breiten} 
 \medskip
 {\footnotesize
  \centerline{Institute of Mathematics}
    \centerline{Technische Universit\"at Berlin} 
    \centerline{Stra\ss e des 17. Juni 136, 10623 Berlin, Germany}
    \centerline{tobias.breiten@tu-berlin.de}
 } 
 \medskip

 \centerline{\scshape Riccardo Morandin}
 \medskip
 {\footnotesize
  \centerline{Institute of Mathematics}
    \centerline{Technische Universit\"at Berlin}
    \centerline{Stra\ss e des 17. Juni 136, 10623 Berlin, Germany}
    \centerline{morandin@math.tu-berlin.de}
 } 

 \medskip
  \centerline{\scshape Philipp Schulze}
 \medskip
 {\footnotesize
 
  \centerline{Institute of Mathematics}
    \centerline{Technische Universit\"at Berlin}
    \centerline{Stra\ss e des 17. Juni 136, 10623 Berlin, Germany}
    \centerline{pschulze@math.tu-berlin.de}
 }

\bigskip

\begin{abstract}
We study linear quadratic Gaussian (LQG) control design for linear port-Hamiltonian systems.
To this end, we exploit the freedom in choosing the weighting matrices and propose a specific choice which leads to an LQG controller which is port-Hamiltonian and, thus, in particular stable and passive.
Furthermore, we construct a reduced-order controller via balancing and subsequent truncation.
This approach is closely related to classical LQG balanced truncation and shares a similar a priori error bound with respect to the gap metric.
By exploiting the non-uniqueness of the Hamiltonian, we are able to determine an optimal pH representation of the full-order system in the sense that the error bound is minimized.
In addition, we discuss consequences for pH-preserving balanced truncation model reduction which results in two different classical $\mathcal{H}_\infty$-error bounds.
Finally, we illustrate the theoretical findings by means of two numerical examples.
\end{abstract}

\section{Introduction}

Many physical processes can naturally be represented as \emph{passive systems}, i.e., dynamical systems that do not internally produce energy. This system class has been studied in great detail in the seminal works \cite{Wil71,Wil72,Wil72a} where various system theoretic characterizations have been given, among them one based on the well-known \emph{Kalman-Yakubovich-Popov} linear matrix inequality (KYP-LMI). Despite its long history, the interest in passivity-based control technique is still unabated, see, e.g., \cite{vdS96,vdSJ14} for a detailed introduction and an overview of existing results. More recently, a particular focus has been on the port-Hamiltonian (pH) representation of passive systems which not only paves the way for an especially targeted analysis of classical control techniques but also for compositional modelling, see, e.g., \cite{DuiMSB09,vdSM13} or specifically robust port-Hamiltonian representations \cite{BeaMV19,MehMS16}.
Port-Hamiltonian modeling has been first developed to provide a unified framework for systems belonging to different physical domains, by using energy as their `lingua franca'.
Even without taking into consideration multi-physics systems, port-Hamiltonian modeling offers several advantages, including passivity and stability properties as consequence of the underlying structure, structure-preserving interconnection allowing for modularized modeling, and structure-preserving methods for space- and time-discretization \cite{BruPM20,Cel17,KotL19,MehM19,HaMS19}.

 While a port-Hamiltonian representation naturally allows for passivity-based control strategies such as \emph{control by interconnection} \cite{vdS96}, it is well-known that general controllers do not preserve the port-Hamiltonian structure. 
 For example, a classical linear quadratic Gaussian (LQG) controller will generally not preserve stability or passivity such that the weighting and covariance matrices have to be modified accordingly, see \cite{Hal94,LozJ88}. 
 The obvious downside of these specific choices is that the original interpretation of the matrices is lost and these rather serve as additional degrees of freedom to preserve the underlying structure. Based on the detailed structural analysis from \cite{Wu16}, in \cite{WuHLM14,Wuetal18} the authors have proposed different choices of weighting matrices in order to enforce LQG controllers to be port-Hamiltonian. 

 Since passive or port-Hamiltonian systems often are inherently infinite-dimensional \cite{JacZ12},  an efficient implementation of control techniques makes the use of reduced-order surrogate models inevitable. 
 For a general introduction to the field of model order reduction, we refer to, e.g., \cite{Ant05a}. 
 Similar to control approaches, classical model reduction techniques will, in general, not preserve the structure within the reduction process. 
 As a remedy, a variety of structure-preserving model reduction approaches have been suggested in the literature. 
 Since an exhaustive overview of all methods is out of the scope of this article, we only refer to \cite{EggKLMM18,Gugetal12,PolvdS12,WolLEK10} which are most relevant for our presentation. 
 Let us however also point to the recent thesis \cite{Lil20} for a more detailed overview of structure-preserving port-Hamiltonian model reduction. 
 Let us also mention that for passive systems, the method of positive real balanced truncation is known to be structure-preserving and, additionally, allows for an a priori error bound \cite{GuiO13,McFDG90}. 
 On the other hand, classical LQG balancing \cite{Cur03,Mey90} will, similar to the controller itself, not lead to a reduced-order port-Hamiltonian system. 
 While remedies exist \cite{PolvdS12,Wuetal18}, enforcing structure-preservation will typically destroy computable a priori error bounds.
This is also true for structure-preserving modifications of the classical balanced truncation method.
In \cite[Section 3.4.1]{Wu16}, the author has derived an error bound for effort-constraint balanced truncation, but the error bound relies on an auxiliary system and cannot be computed a priori. 
Also structure-preserving balanced truncation for the different but still related class of second-order systems, cf.~\cite{ChaLVV06,ReiS08}, lacks computable a priori error bounds.
 
The paper is organized as follows. In Section 2, we recall the necessary background on port-Hamiltonian systems as well as balancing-based model order reduction. Section 3 studies (reduced-order) LQG controller design for port-Hamiltonian systems. In Section 4, we combine results on standard LQG balanced truncation with a recent observation on error bounds in the context of Lyapunov inequalities. Since in the particular case considered here, the particular pH representation is relevant, we further show how to minimize these error bounds in terms of an extremal solution of the KYP-LMI. We briefly discuss some new consequences for classical balanced truncation of port-Hamiltonian systems. Specifically, we derive an error bound for certain spectral factors of the Popov function and subsequently explain the influence of the particular pH representation, leading to a (theoretical) way to minimize these error bounds. Based on some numerical examples, in Section 5 we illustrate our main theoretical findings. 
Appendix A reviews the recent approach from \cite{Wuetal18} and provides two examples showing that the resulting controller will generally not be port-Hamiltonian.
Furthermore, we compare the reduced-order models and controllers obtained by the new approach with the ones obtained by the approach from \cite{Wuetal18} and observe that the reduced-order models are equivalent, whereas the reduced-order controllers differ.

\section{Preliminaries}
\label{sec:preliminaries}

In this section, we collect some well-known results on port-Hamiltonian systems as well as balancing-based model order reduction. For a more detailed discussion of these topics, we refer to other references such as, e.g., \cite{Ant05a,BeaMV19,GugA04,PolvdS12,Wil71,Wil72,Wil72a}.

\subsection{Port-Hamiltonian systems}\label{subsec:pH}

Let us consider a pH system of the form 
\begin{equation}\label{eq:pH}
\begin{aligned} 
  \dot{x}&= \underbrace{(J-R)Q}_{:=A}x+ Bu,\quad x(0)=0,    \\ 
  y&=\underbrace{B^\top Q}_{:=C} x ,
\end{aligned}
\end{equation} 
where $J,R,Q\in \mathbb R^{n\times n}$, $B\in \mathbb R^{n \times m}$ and $ J^\top = -J$, $R=R^\top \succeq 0 $, $Q=Q^\top\succ 0$, together with a Hamiltonian function $\mathcal{H}(x)=\frac{1}{2}x^\top Qx$. Throughout this article, we assume the system \eqref{eq:pH} to be minimal, i.e., the matrix pairs $(A,B)$ and $(A,C)$ to be controllable and observable, respectively. In particular, this implies that $(A,B)$ and $(A,C)$ are stabilizable and detectable, respectively.
It is well-known that, by taking the Hamiltonian as an energy storage function, for $t_1\ge 0 $ the following \emph{dissipation inequality} holds:
\begin{equation}
    \label{eq:dissipationIneq}
    \mathcal{H}(x(t_1))-\mathcal{H}(x(0 )) = \int_{0 }^{t_1}\big(y(s)^\top u(s)-x(s)^\top QRQx(s) \big)\, \mathrm{d}s \le \int_{0 }^{t_1} y(s)^\top u(s) \, \mathrm{d}s,
\end{equation}
for arbitrary trajectories $x(\cdot)$ of \eqref{eq:pH}. Systems which possess a storage function satisfying such an inequality are called \emph{passive} and cannot produce energy internally and, in particular, are  stable (though not necessarily asymptotically stable).
Furthermore, passive systems are \emph{dissipative} with supply rate $y^\top u$, see \cite{Wil72,Wil72a} for more details on dissipative systems.
Another important property of port-Hamiltonian systems with Hamiltonian function $\frac{1}{2}x^\top Qx$ is that $X=Q$ is a solution to the KYP-LMI:
\begin{align}\label{eq:kyp-lmi}
W(X)=\begin{bmatrix} -A^\top X - XA & C^\top - XB \\ C - B^\top X & 0 \end{bmatrix} \succeq 0, \quad X=X^\top \succeq 0,
\end{align}
since
\begin{align*}
    W(Q) =\begin{bmatrix} -A^\top Q - QA & C^\top - QB \\ C - B^\top Q & 0 \end{bmatrix} 
    = \begin{bmatrix} 2QRQ & 0 \\ 0 & 0 \end{bmatrix} \succeq 0.
\end{align*}
On the other hand, if $X=X^\top\succ 0$ is a solution of the KYP-LMI \eqref{eq:kyp-lmi}, then the system \eqref{eq:pH} can be equivalently written as
\begin{equation}\label{eq:pHX}
\begin{aligned} 
  \dot{x}&= (J_X-R_X)Xx+ Bu,\quad x(0)=0,    \\ 
  y&= B^\top X x ,
\end{aligned}
\end{equation} 
where $J_X=\frac{1}{2}(AX^{-1}-X^{-1}A^\top)=-J_X^\top$ and $R_X=-\frac{1}{2}(AX^{-1}+X^{-1}A^\top)=R_X^\top\succeq 0$.
This system is thus again port-Hamiltonian, but with respect to the Hamiltonian function $\mathcal H_X(x)=\frac{1}{2}x^\top Xx$, that is in general different from $\mathcal H(x)$.
Therefore, the representation of passive systems via a pH structure \eqref{eq:pH} is not unique.
In particular, there exist minimal $X_{\mathrm{min}}$ and maximal $X_{\mathrm{max}}$ solutions s.t.\@ for every solution $X$ to \eqref{eq:kyp-lmi}, it holds that $0\preceq X_{\mathrm{min}}\preceq X \preceq X_{\mathrm{max}}$.
While the system is equivalent, the new Hamiltonian $\mathcal H_X(x)$ may not have the same relevance as one associated with a particular solution of \eqref{eq:kyp-lmi}. For example, in \cite{BeaMV19} the authors investigate a maximally robust (w.r.t.\@ the passivity radius) representation based on what is called the \emph{analytic center}. As will be shown later, for the purpose of model reduction, the extremal solutions $X_{\min}$ and $X_{\max}$ are of special interest.

Another important property of port-Hamiltonian systems is that they maintain their structure when a state space transformation is applied.
In fact, if $T\in\mathbb{R}^{n\times n}$ is an invertible matrix, then by applying the change of variable $\tilde x=Tx$ we obtain the equivalent system
\begin{equation}\label{eq:pHCoV}
\begin{aligned} 
  \dot{\tilde{x}} &= (\tilde{J}-\tilde{R})\tilde{Q}\tilde{x} + \tilde{B}u,\quad \tilde x(0)=0,    \\ 
  y&= \tilde B^\top \tilde Q \tilde x ,
\end{aligned}
\end{equation} 
where $\tilde J=TJT^\top=-\tilde J^\top$, $\tilde R=TRT^\top=\tilde R^\top\succeq 0$, $\tilde Q=T^{-\top}QT^{-1}=\tilde Q^\top\succ 0$, and $\tilde B=TB$.
We will then have $\tilde A=(\tilde J-\tilde R)\tilde Q=TAT^{-1}$.
Note that the Hamiltonian function $\tilde{\mathcal H}(\tilde x)=\frac{1}{2}\tilde x^\top\tilde Q\tilde x=\frac{1}{2}x^\top Qx$ is still the same.
Thus, state space transformations do not destroy the port-Hamiltonian structure and we will make use of this property when discussing balancing methods in the upcoming sections.

Similarly, a Petrov--Galerkin projection $\mathbb P=VW^\top$ of the form $A_p=W^\top AV$, $B_p=W^\top B$, $C_p=CV$ with $QV=WQ_p$ for a matrix $Q_p=Q_p^\top\succ 0$ and $W^\top V=I_p$ leads to a projected port-Hamiltonian system of the form
\begin{equation}\label{eq:pHPG}
\begin{aligned} 
  \dot x_p &= (J_p-R_p)Q_p x_p + B_p u,\quad x_p(0)=0,    \\ 
  y&= B_p^\top Q_p x_p ,
\end{aligned}
\end{equation} 
where $x\approx Vx_p$, $J_p=W^\top JW$, $R_p=W^\top RW$, and $B_p=W^\top B$, with Hamiltonian function $\mathcal H_p(x_p)=\frac{1}{2}x_p^\top Q_px_p$.

Due to the minimality of \eqref{eq:pH}, passivity is equivalent to \emph{positive realness}, i.e., the transfer function $G(s)=C(sI_n-A)^{-1}B$ is analytic in the open right half plane $\mathbb C_+$ and satisfies
\begin{align}\label{eq:popov_fct}
G(s)^* + G(s)\succeq 0 \quad \text{for all } s \in \mathbb C_+.
\end{align}
If all eigenvalues of $A$ are in the open left half plane, then positive realness can be equivalently characterized by the positive semidefiniteness  of the so-called Popov function $\Phi(s)=G(-s)^\top+G(s)$ on the imaginary axis.
Let us consider the following factorization of the Popov function $\Phi(s)$ (see, e.g., \cite{Wil72a}):
\begin{align}\label{eq:popov_aux}
 \Phi(s)= \begin{bmatrix} B^\top (-sI_n-A^\top)^{-1} & I_m \end{bmatrix} W(X) \begin{bmatrix} (sI_n-A)^{-1}B \\ I_m \end{bmatrix}.
\end{align}
In particular, if $X$ is a solution to \eqref{eq:kyp-lmi}, then there exists $L_X\in \mathbb R^{n\times k},k\le n$ such that 
\begin{align}\label{eq:popov_factor}
    \Phi(s)= \begin{bmatrix} B^\top (-sI_n-A^\top)^{-1} & I_m \end{bmatrix} \begin{bmatrix} L_X^\top \\ 0 \end{bmatrix} \begin{bmatrix} L_X & 0 \end{bmatrix} \begin{bmatrix} (sI_n-A)^{-1}B \\ I_m \end{bmatrix}.
\end{align}
We will later on focus on a specific factorization of the form \eqref{eq:popov_factor} that will be associated with
\begin{equation}
    \label{eq:dualKYPLMI}
    \begin{bmatrix} 
        -AY - YA^\top & B-YC^\top  \\
        B^\top - CY & 0 
    \end{bmatrix} 
    \succeq 0,\quad Y=Y^\top \succeq 0,
\end{equation}
i.e., the dual version of \eqref{eq:kyp-lmi}.

\subsection{Balancing-based and effort-constraint model order reduction} \label{subsec:effortConstraint}
 
 The concept of system balancing goes back to \cite{Moo81,MulR76} and relies on the infinite-time  controllability and observability Gramians $\Lcalc,\Mcalo$ associated with a linear (asymptotically) stable system characterized by $(A,B,C)$. These Gramians satisfy the following Lyapunov equations:
 \begin{equation}\label{eq:Lyap_con_obs}
 \begin{aligned}
 A\Lcalc + \Lcalc A^\top + BB^\top &=0, \\
 A^\top\Mcalo + \Mcalo A + C^\top C&=0.
 \end{aligned}
 \end{equation}
 The main idea of balancing-based model order reduction now is to find a specific state space transformation $(A,B,C)\leadsto (TAT^{-1},TB,CT^{-1})$ that allows one to simultaneously measure the amount of controllability and observability in the system. This is possible since in the new coordinates, the Gramians are given by $T\Lcalc T^\top$ and $T^{-\top}\Mcalo T^{-1}$, allowing one to construct an appropriate contragredient transformation such that 
 \begin{align*}
     T\Lcalc T^\top=T^{-\top}\Mcalo T^{-1}= \Sigma= \mathrm{diag}(\sigma_1,\dots,\sigma_n).
 \end{align*}
In fact, the so-called square root balancing method (see, e.g., \cite{Ant05a,GugA04}) yields such a transformation via:
\begin{align*}
    T=\Sigma^{-\frac{1}{2}} Z^\top L_{\Mcalo}, \quad T^{-1}=L_{\Lcalc}^\top U \Sigma^{-\frac{1}{2}},
\end{align*}
where we have the following singular value and Cholesky decompositions
\begin{align*}
L_{\Lcalc} L _{\Mcalo}^\top=U\Sigma Z^\top ,\quad     \Lcalc = L_{\Lcalc}^\top L_{\Lcalc}, \quad \Mcalo = L_{\Mcalo}^\top L_{\Mcalo}.
\end{align*}
For an already balanced model $(A,B,C)$, consider then the partitioning:
\begin{align*}
    A=\begin{bmatrix} A_{11} & A_{12} \\ A_{21} & A_{22} \end{bmatrix}, \quad 
        B=\begin{bmatrix} B_{1} \\ B_{2} \end{bmatrix},\quad  C=\begin{bmatrix} C_1 & C_2 \end{bmatrix}.
\end{align*}
A reduced-order model $(A_r,B_r,C_r)$ is constructed by truncation, i.e., choosing $(A_{11},B_1,C_1)$. In the particular case of \eqref{eq:pH}, a straightforward implementation of this approach will generally not preserve the port-Hamiltonian structure. As a remedy, in \cite{PolvdS12}, the authors have proposed a so-called \emph{effort-constraint} reduction method which is based on a (balanced) partitioning of \eqref{eq:pH} of the form:
\begin{align}\label{eq:part}
       \Jb=\begin{bmatrix} J_{11} & J_{12} \\ J_{21} & J_{22} \end{bmatrix}, \quad 
       \Rb=\begin{bmatrix} R_{11} & R_{12} \\ R_{21} & R_{22} \end{bmatrix}, \quad 
       \Qb=\begin{bmatrix} Q_{11} & Q_{12} \\ Q_{21} & Q_{22} \end{bmatrix}, \quad 
        \Bb=\begin{bmatrix} B_{1} \\ B_{2} \end{bmatrix}.
\end{align}
The reduced-order model is then defined as 
\begin{align}\label{eq:eff_const}
 J_r = J_{11},\quad R_r = R_{11}, \quad Q_r=Q_{11}-Q_{12}Q_{22}^{-1} Q_{21}, \quad B_r = B_1, \quad C_r=B_r^\top Q_r.
\end{align}
Note in particular that with $Q=Q^\top \succ 0$ it also holds that the Schur complement is positive definite, i.e., $Q_r=Q_r^\top \succ 0$.
Furthermore, it can be noted that this reduced-order model can be obtained via Petrov--Galerkin projection $\mathbb{P}=\Vb \Wb^\top$ applied to the balanced system, with $\Wb^\top=\begin{bmatrix}I_r & 0\end{bmatrix}$ and $\Vb^\top=\begin{bmatrix}I_r & -Q_{12}Q_{22}^{-1}\end{bmatrix}$, satisfying $\Wb^\top \Vb=I_r$ and $\Qb \Vb=\Wb Q_r$, or equivalently a Petrov--Galerkin projection $\mathbb{P}=VW^\top$ applied to the original system, with $W^\top=\Wb^\top T$ and $V=T^{-1}\Vb$ (see \autoref{subsec:pH}).
Hence, by construction the reduced-order model is port-Hamiltonian. This structure-preservation however comes with the loss of a (classical) $\mathcal{H}_\infty$-error bound. Let us remark that in \cite{Wuetal18}, the authors have derived an a posteriori error bound which makes use of an auxiliary system. Under certain assumptions on $R$ and $B$, in Section \ref{sec:classic_bt}, we modify the balancing approach such that an a priori error bound, analogous to the classical one, is valid. Furthermore, we establish a connection to balancing of a spectral factorization of the Popov function. Finally, we mention that the method of positive real balanced truncation \cite{HarJS83} aims at the preservation of passivity and relies on balancing the (dual) extremal solutions to \eqref{eq:kyp-lmi}. 

\section{Port-Hamiltonian reduced LQG control design}\label{sec:pH-LQG}

In this section, given a port-Hamiltonian system \eqref{eq:pH}, 
our interest is the design of a controller 
\begin{equation}\label{eq:controller}
 \begin{aligned}
   \xdotc &= \Ac x _c + \Bc \uc , \quad \xc(0)=0, \\
   \yc &= \Cc \xc,
 \end{aligned}
\end{equation}
with $(\Ac,\Bc,\Cc)\in \mathbb R^{\nc\times \nc}\times \mathbb R^{\nc \times m} \times \mathbb R^{m \times \nc}$ such that the dynamics $x(\cdot),\xc(\cdot)$ satisfy a desired behavior, e.g., are associated with the solution of an optimal control problem.
It is well-known, see, e.g., \cite[Section 7]{vdS96}, that if $(\Ac,\Bc,\Cc)$ is a pH system, the resulting closed loop system allows for a pH formulation if a \emph{power-conserving interconnection} of the form $u=-\yc$ and $\uc=y$ is used. Note however that classical LQG control design would result in
\begin{align}\label{eq:LQG_controllers}
\Ac = A-B\tilde{\mathcal R}^{-1}B^\top \Pc-\Pf C^\top \Rf^{-1}C , \quad \Bc=\Pf C^\top \Rf^{-1}, \quad \Cc=\tilde{\mathcal R}^{-1} B^\top \Pc
\end{align}
where $\Pc $ and $\Pf$ are solutions to the following control and filter Riccati equations
\begin{align}
A^\top \Pc +\Pc A - \Pc B\tilde{\mathcal R}^{-1} B^\top \Pc + \tilde{\mathcal Q}&=0 , \label{eq:lqg_care} \\
 A \Pf + \Pf A^\top - \Pf C^\top \Rf ^{-1} C \Pf + \Qf &=0 \label{eq:lqg_fare}.
\end{align}
Since such  controllers  generally do not preserve the pH structure (for a stability violation, we refer to \cite{Joh79}), we suggest a particular choice of the weighting matrices $\tilde{\mathcal R},\Rf,\tilde{\mathcal Q}$ and $\Qf$ and consider the solutions $\Phatc$ and $\Phatf$ of
\begin{align}
A^\top \Phatc +\Phatc A - \Phatc B  B^\top \Phatc + C^\top C&=0 , \label{eq:lqg_care_ph} \\
 A \Phatf + \Phatf A^\top - \Phatf C^\top C \Phatf + BB^\top +2R &=0 \label{eq:lqg_fare_ph}.
\end{align}
The choice $\tilde{\mathcal R}=\Rf=I$ is mainly for simplicity and could be taken into account when a scaling of the input weights would be of further interest. Let us emphasize that our method is inspired by the approach proposed in \cite{Wuetal18} and the above modification of the LQG weights is based on ideas similar to those in \cite{LozJ88,Wu16}. In Appendix \ref{apdxA} we review the method from \cite{Wuetal18} and provide two examples which show that the resulting controller may generally not be realized as a port-Hamiltonian system.

With the particular relation between the covariance and weighting matrices from \eqref{eq:lqg_care_ph} and \eqref{eq:lqg_fare_ph}, the solution $\Phatc$ satisfies the KYP-LMI \eqref{eq:kyp-lmi} such that the controller is port-Hamiltonian.

\begin{theorem}\label{thm:cont_is_pH}
 Let a minimal port-Hamiltonian system \eqref{eq:pH} be given by $(J,R,Q,B).$ If $\Phatc$ and $\Phatf$ are the unique stabilizing solutions of \eqref{eq:lqg_care_ph} and \eqref{eq:lqg_fare_ph}, respectively, then the associated LQG controller defined by 
 \begin{align*}
     \Ahatc=A-BB^\top\Phatc-\Phatf C^\top C , \ \ \Bhatc = \Phatf C^\top, \ \ \Chatc=B^\top \Phatc  
 \end{align*}
 is port-Hamiltonian. 
 In particular, with the Hamiltonian function $\Hc(\xc)=\frac{1}{2}\xc^\top \Phatc \xc,$ a pH realization of the controller is as follows:
 \begin{align*}
 \Jhatc = \tfrac{1}{2}(\Ahatc \Phatc^{-1}-\Phatc^{-1}\Ahatc^\top ),\quad \Rhatc = \tfrac{1}{2}(\Phatc^{-1} Q + I_n)BB^\top (\Phatc^{-1} Q + I_n)^\top,\quad \Qhatc = \Phatc,\quad \Bhatc =B.
 \end{align*}
\end{theorem}

\begin{proof}
 Note that due to the stabilizability and detectability of the system, the equations \eqref{eq:lqg_care_ph} and \eqref{eq:lqg_fare_ph} have unique stabilizing solutions $\Phatc$ and $\Phatf$, respectively. In fact, it holds that $\Phatf=Q^{-1}$ since
 \begin{align*}
     &AQ^{-1} + Q^{-1} A^\top - Q^{-1} C^\top C Q^{-1} +2R+BB^\top \\
     &= (J-R)QQ^{-1} + Q^{-1}Q^\top (J^\top -R^\top) - Q^{-1}QBB^\top Q Q^{-1}+2R +BB^\top =0.
 \end{align*}
 This, however, justifies the following derivations
 \begin{align*}
     \Ahatc^\top \Phatc + \Phatc \Ahatc &= (A-BB^\top\Phatc-\Phatf C^\top C)^\top \Phatc + \Phatc (A-BB^\top\Phatc-\Phatf C^\top C) \\
     &= A^\top \Phatc + \Phatc A - \Phatc BB^\top \Phatc - \Phatc BB^\top \Phatc - C^\top C \Phatf \Phatc - \Phatc \Phatf C^\top C \\
     &=-C^\top C -\Phatc BB^\top \Phatc - C^\top C Q^{-1}\Phatc - \Phatc Q^{-1} C^\top C \\
     &= - (C^\top +\Phatc B)(C^\top +\Phatc B)^\top \preceq 0.
 \end{align*}
 We further obtain that 
 \begin{align*}
     \Bhatc^\top \Phatc = (\Phatf C^\top )^\top \Phatc = C Q^{-1} \Phatc = B^\top \Phatc = \Chatc.
 \end{align*}
 We conclude that $\Ahatc^\top \Phatc + \Phatc \Ahatc\preceq 0$ as well as $  \Bhatc^\top \Phatc=\Chatc$ such that $\Phatc=\Phatc^\top \succeq 0$ satisfies 
 \begin{align*}
     \begin{bmatrix} -\Ahatc ^\top \Phatc -\Phatc \Ahatc & \Chatc^\top -\Phatc \Bhatc \\ 
      \Chatc-\Bhatc^\top \Phatc & 0 \end{bmatrix} \succeq 0 .
 \end{align*}
 Since $(A,C)$ is observable, it further holds that $\Phatc \succ 0$. Due to the discussion in \autoref{subsec:pH}, we can conclude that $(\Ac,\Bc,\Cc)$ is port-Hamiltonian with respect to the Hamiltonian function $\mathcal H_{\mathrm{c}}(\xc)=\frac{1}{2}\xc^\top\Phatc\xc$.
 In particular, the matrices $\Jhatc$ and $\Rhatc$ can be constructed as outlined in \autoref{subsec:pH}.
\end{proof}

\begin{remark}\label{rem:rel_between_Pf}
Note that since $\Phatf=Q^{-1}$ the controller is of the form:
 \begin{align*}
     \xdotc &= (A-BB^\top \Phatc- BC )\xc +B \uc, \quad \xc(0)=0, \\
     \yc&= B^\top \Phatc \xc.
 \end{align*}
In particular, the system matrix $\Ahatc$ combines the feedback structure  of a classical linear quadratic regulator $\uc=-BB^\top \Phatc \xc$ with that of a simple output feedback $\uc=-BC\xc$.
\end{remark}
 
\begin{remark}
    \label{rem:modifyCARE}
    Theorem~\ref{thm:cont_is_pH} states that the LQG controller based on the control Riccati equation \eqref{eq:lqg_care_ph} and the modified filter Riccati equation \eqref{eq:lqg_fare_ph} is port-Hamiltonian.
    We note that a similar result can be obtained by modifying the control Riccati equation.
    More precisely, by choosing the weighting matrices $\tilde{\mathcal Q} = C^\top C+2QRQ$ and $\Qf = BB^\top$, the resulting Gramian $\Pc$ is simply given by $\Pc=Q$ and the resulting LQG controller is port-Hamiltonian with Hamiltonian $\frac12 \xc^\top \Pf^{-1} \xc$.
    This choice of the weighting matrix $\tilde{\mathcal Q} = C^\top C+2QRQ$ also permits a physical interpretation as it corresponds to an optimal control problem where the LQR cost is given by the sum of the squared input norm, the squared output norm, and (twice) the dissipated energy, cf.~\eqref{eq:dissipationIneq}.
    This cost function leads to an LQR control which is simply given by the output feedback $u = -B^\top\Pc x = -B^\top Qx = -y$.
\end{remark}

\subsection{Structure-preserving LQG balanced truncation for port-Hamiltonian systems}
\label{subsec:modifiedLQG}

With regard to a numerical implementation for large-scale systems, let us further derive a reduced port-Hamiltonian controller which replaces $(\Ahatc,\Bhatc,\Chatc)$ by a surrogate model. This is done in three steps: first, the system is transformed into a balanced form \eqref{eq:part} such that $\Phatf=\Phatc=\Sigma=\mathrm{diag}(\sigma_1,\dots,\sigma_n).$ 
Subsequently, a reduced-order model is obtained by the effort-constraint method:
\begin{equation}\label{eq:rom_new}
 \begin{aligned}
   \dot{\widehat{x}}_{r} &= (\widehat{J}_r-\widehat{R}_r)\widehat{Q}_r \widehat{x}_{r}+\widehat{B}_ru,\quad \widehat{x}_{r}(0)=0, \\
   \widehat{y}_r&= \widehat{C}_r \widehat{x}_{r},
 \end{aligned}
\end{equation}
where $(\widehat{J}_r,\widehat{R}_r,\widehat{Q}_r,\widehat{B}_r,\widehat{C}_r)$ are as in \eqref{eq:eff_const}. Finally, a reduced controller $(\widehat{A}_{\mathrm{c}r},\widehat{B}_{\mathrm{c}r},\widehat{C}_{\mathrm{c}r})$ is constructed according to Theorem \ref{thm:cont_is_pH}. The entire procedure is summarized in Algorithm \ref{alg:new}.

\begin{algorithm}[H]
    \caption{pH-preserving LQG reduced controller design} 
    \label{alg:new}
    \begin{algorithmic}[1]
      \REQUIRE $J, R,Q\in \mathbb R^{n\times n}, B\in \mathbb R^{n\times m}$ as in \eqref{eq:pH} with minimal $(A,B,C)$
      \ENSURE Reduced-order LQG controller $(\widehat{A}_{\mathrm{c}r},\widehat{B}_{\mathrm{c}r},\widehat{C}_{\mathrm{c}r})$ \\
      \STATE Compute $\Phatc$ and $\Phatf=Q^{-1}$ solving \eqref{eq:lqg_care_ph} and \eqref{eq:lqg_fare_ph}.\!\!\!
      \STATE Compute $T\in \mathbb R^{n\times n}$ defined by $T\Phatf T^\top = T^{-\top} \Phatc T^{-1} = \mathrm{diag}(\sigma_1,\dots,\sigma_n)$.
      \STATE Balance the system $\Jb=TJT^\top , \Rb = TRT^\top , \Qb =T^{-\top } Q T^{-1} $ and $\Bb=TB$.
      \STATE Obtain the reduced system $(\widehat{A}_r,\widehat{B}_r,\widehat{C}_r)$ by the effort-constraint method. 
      \STATE Compute the reduced-order LQG controller $(\widehat{A}_{\mathrm{c}r},\widehat{B}_{\mathrm{c}r},\widehat{C}_{\mathrm{c}r})$ based on Theorem \ref{thm:cont_is_pH} for $(\widehat{A}_r,\widehat{B}_r,\widehat{C}_r)$.
      \end{algorithmic}
\end{algorithm}

It turns out that due to the particular choice of the weighting matrix $\Qf = BB^\top +2R$ and the resulting relation $\Phatf=Q^{-1}$, the reduced-order model \eqref{eq:rom_new} can equivalently be obtained by simple truncation.

\begin{theorem}\label{thm:rom_ph}
    Let a minimal port-Hamiltonian system \eqref{eq:pH} be given by $(J,R,Q,B)$ and let $(\widehat{A}_r,\widehat{B}_r,\widehat{C}_r)$ be the corresponding reduced-order model obtained in Algorithm \ref{alg:new}. 
    Then $(\widehat{A}_r,\widehat{B}_r,\widehat{C}_r)$ has a port-Hamiltonian realization. 
    Moreover, $(\widehat{A}_r,\widehat{B}_r,\widehat{C}_r)$ can be obtained by simple truncation of a system that is balanced w.r.t.~$\Phatc$ and $\Phatf$ as in \eqref{eq:lqg_care_ph} and \eqref{eq:lqg_fare_ph}, respectively.
\end{theorem}
\begin{proof}
 The fact that $(\widehat{A}_r,\widehat{B}_r,\widehat{C}_r)$ has a port-Hamiltonian realization immediately follows from its construction by the effort-constraint reduction technique. Note that for the balanced model it holds that $\Phatc=\Phatf=\Sigma$. As mentioned in the proof of Theorem \ref{thm:cont_is_pH}, we have that $\Phatf=Q_{\mathrm{b}}^{-1}$, i.e., $Q_{\mathrm{b}}=\mathrm{diag}(\frac{1}{\sigma_1},\dots,\frac{1}{\sigma_n})$. Hence, we conclude that $Q_{\mathrm{b}}$ is diagonal such that the Schur complement in \eqref{eq:eff_const} reads $Q_r=Q_{11}-Q_{12}Q_{22}^{-1}Q_{21}=Q_{11}$ which shows the second assertion.
\end{proof}

The relation $\Phatf=Q^{-1}$ yields some additional beneficial implications with regard to the computation of the reduced-order matrices $\widehat{J}_r$, $\widehat{R}_r$ and $\widehat{Q}_r$.
According to the general idea of square root balancing, let us assume that we have the decompositions 
\begin{align*}
    T=\Sigma^{-\frac{1}{2}} Z^\top L_{\Phatc}, \quad T^{-1} = L_{\Phatf}^\top U \Sigma^{-\frac{1}{2}},
\end{align*}
where 
\begin{align}\label{eq:aux1}
    L_{\Phatf} L _{\Phatc}^\top = U\Sigma Z^\top , \quad Q^{-1}=\Phatf = L_{\Phatf}^\top L _{\Phatf},\quad 
    \Phatc = L_{\Phatc}^\top L _{\Phatc}.
\end{align}
The balanced system then is given by applying the change of variable $\xb=Tx$ as in \eqref{eq:pHCoV}, leading to $\Ab=TAT^{-1}$, $\Bb=TB$, $\Cb=CT^{-1}$ and the port-Hamiltonian formulation: 
\begin{align*}
    \Jb = TJT^\top, \quad \Rb = TRT^\top, \quad \Qb = T^{-\top}QT^{-1} = \Sigma^{-1}.
\end{align*}
We can therefore construct the reduced pH formulation by the Petrov--Galerkin projection $\mathbb P=VW^\top$ with $V^\top=\begin{bmatrix}I_r & 0\end{bmatrix}T^{-\top}$ and $W^\top=\begin{bmatrix}I_r & 0\end{bmatrix}T$ applied to $(A,B,C)$, i.e., 
\begin{equation}\label{eq:pH_red}
    \begin{alignedat}{3}
    \widehat{A}_r &= W^\top AV, \quad& \widehat{B}_r &= W^\top B, \quad& \widehat{C}_r &= CV, \\
    \widehat{J}_r &= W^\top J W, \quad& \widehat{R}_r &= W^\top R W, \quad& \widehat{Q}_r &= (W^\top Q^{-1} W)^{-1} = \Sigma_1^{-1}.
    \end{alignedat}
\end{equation}
This representation will be of particular interest in section \ref{sec:classic_bt}, where we provide an error bound that involves a reduced factorization of $\widehat{R}_r$, cf.\@ Remark \ref{rem:compare_spec_fac}.

\begin{remark}\label{rem:con_to_LQG}
With the previous considerations, we conclude that the reduced-order model $(\widehat{A}_r,\widehat{B}_r,\widehat{C}_r)$ results (and itself is balanced) from the truncation of a system that is balanced w.r.t.\@ the Gramians $\Phatc$ and $\Phatf$, respectively. 
We point out that the representation corresponds to a rather simple modification of classical LQG balanced truncation (by adding the term $2R$) which will be utilized below to derive a suitably modified error bound in the gap metric.
\end{remark}

\subsection{PH formulation in co-energy variables}\label{sec:generalizedStateSpace}

For many applications, a pH formulation as in \eqref{eq:pH} may lead to limitations and it often turns out to be beneficial to consider a differential-algebraic formulation. 
While a detailed analysis of the differential-algebraic case is out of the scope of this article, below we discuss required modifications of the Gramians defined in \eqref{eq:lqg_care_ph} and \eqref{eq:lqg_fare_ph} in the particular case of a formulation in the \emph{co-energy variables} $z\coloneqq Qx$. 
Starting from the description in \eqref{eq:pH}, this leads to the formulation
\begin{equation}\label{eq:pH_gen}
\begin{aligned}
E\dot{z}&= (J-R)z+Bu, \ \ z(0)=0,\\
y&= B^\top z,
\end{aligned}
\end{equation}
of an implicitly given pH system, where $E=Q^{-1}=E^\top\succ 0$, with the Hamiltonian function $\tilde{\mathcal H}(z)=\frac{1}{2}z^\top Ez$. Note in particular that this formulation has the additional advantage of being linear in the defining matrices $(E,J,R).$ Based on well-known LQG control theory for generalized state space systems (e.g., \cite{Meh91}), it is natural to replace \eqref{eq:lqg_care_ph} by 
\begin{align}\label{eq:lqg_care_ph_gen}
 (J-R)^\top \Ptildec E + E^\top\Ptildec (J-R) - E ^\top\Ptildec BB^\top \Ptildec E + BB^\top  &=0.
\end{align}
With regard to port-Hamiltonian structure of a controller, instead of \eqref{eq:lqg_fare_ph} we consider 
\begin{align}\label{eq:lqg_fare_ph_gen}
 (J-R) \Ptildef E^\top + E \Ptildef (J-R)^\top - E \Ptildef BB^\top \Ptildef E^\top + BB^\top +2R &=0.
\end{align}
On the one hand, after left- and right-multiplication of \eqref{eq:lqg_care_ph_gen} by $Q$, we get again \eqref{eq:lqg_care_ph}, thus $\Ptildec=\Phatc$.
On the other hand, it is clear that $\Ptildef=Q=E^{-1}$ is the solution of \eqref{eq:lqg_fare_ph_gen}.

With a slight modification of the arguments used in the proof of Theorem \ref{thm:cont_is_pH}, we have the following result.  
\begin{corollary}\label{cor:cont_is_pH_gen}
 Let $(E,J,R,B)$ define a minimal implicit port-Hamiltonian system of the form \eqref{eq:pH_gen}. If $\Ptildec$ and $\Ptildef$ are the unique stabilizing solutions of \eqref{eq:lqg_care_ph_gen} and \eqref{eq:lqg_fare_ph_gen}, respectively, then the associated LQG controller defined by 
 \begin{align*}
 \Etildec = \Ptildec^{-1}, \ \ \Atildec=\Ptildef^{-1}E^{-1}(J-R-BB^\top\Ptildec E-E\Ptildef BB^\top)E^{-1}\Ptildec^{-1}, \ \ \Btildec =B, \ \ \Ctildec=B^\top  
 \end{align*}
 is a port-Hamiltonian system in co-energy variable formulation.
\end{corollary}

\begin{proof}
 Since $\widetilde E_c=\Ptildec^{-1}=\widetilde E_c^\top\succ 0$, we only have to prove that $\Atildec+\Atildec^\top\preceq 0$. Since $\Ptildef=E^{-1}$, one easily deduces that
 \begin{align*}
     E\Ptildec(\Atildec+\Atildec^\top)\Ptildec E
     &= E\Ptildec(J-R) + (J-R)^\top\Ptildec E -2E\Ptildec BB^\top\Ptildec E-E\Ptildec BB^\top-BB^\top\Ptildec E = \\
     &= -E\Ptildec BB^\top\Ptildec E-BB^\top-E\Ptildec BB^\top-BB^\top\Ptildec E = \\
     &= -(I_n+E\Ptildec)BB^\top(I_n+E\Ptildec)^\top\preceq0,
 \end{align*}
 therefore $\Atildec+\Atildec^\top\preceq 0$, as requested.
\end{proof}
Alternatively to the proof of Corollary \ref{cor:cont_is_pH_gen}, it can be easily shown that the associated LQG controller corresponds to a co-energy formulation of the LQG controller that we constructed in Theorem~\ref{thm:cont_is_pH}. 
Thus, it is clearly port-Hamiltonian.

\section{An error bound for pH-LQG reduced-order controllers}

As mentioned in Remark \ref{rem:con_to_LQG}, the reduced model from Algorithm \ref{alg:new} can be interpreted as a variant of LQG balanced truncation where the constant term $BB^\top$ is extended by the term $2R.$
It is thus obvious to study possible error bounds w.r.t.\@ the gap metric. Some considerations in this direction have been given in \cite{Wu16}. In contrast to the latter work, here we follow the reasoning in \cite{DamB14} and therefore aim at showing that the reduced Gramians satisfy two Lyapunov \emph{inequalities} which will yield an error bound (for the closed loop dynamics). 
Note that the idea of using Lyapunov inequalities instead of Lyapunov equations to derive an error bound can already be found for linear time-varying systems in \cite{SanR04}.
From now on we assume that $\Sigma$ is the balanced Gramian solving the equations \eqref{eq:lqg_care_ph} and \eqref{eq:lqg_fare_ph} such that the reduced-order model $(\widehat{A}_r,\widehat{B}_r,\widehat{C}_r)$ results from truncation of the associated balanced system $(A,B,C)$.

\subsection{Error bounds for standard LQG balanced truncation}
\label{sec:standardLQGBT}

While the gap metric is standard in the context of LQG balanced truncation, cf.~\cite{Cur03,McFDG90,Mey90,MoeRS11}, for a self-contained presentation, let us recall some well-known concepts regarding system norms and coprime factorizations that can be found in, e.g., \cite{Cur90,McFDG90,SefO93,Vid84,ZhoDG96}. The Hardy spaces $\mathcal{H}_\infty^{p,m}$ and $\mathcal{H}_2^{p,m}$ are defined by
\begin{align*}
    \mathcal{H}_2^{p,m} &:=\left\{ F\colon \mathbb C^+ \to \mathbb C^{p\times m} \ | \ F \text{ is analytic, } \| F\|_{\mathcal{H}_2}:=\left( \sup\limits_{\sigma>0} \int _{-\infty}^{\infty} \| F(\sigma + \imath \omega )\| _{\mathrm{F}}^2 \, \mathrm{d}\omega \right)^{\frac{1}{2}} < \infty \right\}, \\
    \mathcal{H}_{\infty}^{p,m} &:=\left\{ F\colon \mathbb C^+ \to \mathbb C^{p\times m} \ | \ F \text{ is analytic, } \| F\|_{\mathcal{H}_\infty}:=\sup\limits_{z \in \mathbb C^+}  \| F(\sigma + \imath \omega )\| _2 < \infty \right\}.
\end{align*}
Let us further introduce $\mathcal{R}\mathcal{H}_{\infty}^{p,m}$ and $\mathcal{R}\mathcal{H}_{\infty}^{p,m}$ consisting of matrix valued real rational functions that additionally are in $\mathcal{H}_{\infty}^{p,m}$ and $\mathcal{H}_{2}^{p,m}$, respectively.
Based on the stabilizing solution of \eqref{eq:lqg_care_ph} we may construct a normalized right coprime factorization of the form
\begin{align}\label{eq:left_coprime_G}
    G(s)=C(sI_n-A)^{-1}B=N(s)M(s)^{-1},
\end{align}
where $M\in \mathcal{R}\mathcal{H}_\infty^{m,m}$, $N\in \mathcal{R}\mathcal{H}_2^{m,m}$ and the (control) closed loop matrix $A_{\Phatc}$ are given as follows:
\begin{align*}
 N(s)=C(sI_n-A_{\Phatc})^{-1}B, \quad M(s)=I_m-B^\top \Phatc(sI_n-A_{\Phatc})^{-1} B, \quad A_{\Phatc} = A-BB^\top \Phatc.
\end{align*}
Here, right coprimeness means that there exist transfer functions $X,Y\in \mathcal{R} \mathcal{H}_\infty^{m,m}$ such that 
\begin{align*}
 X(s)M(s)+Y(s)N(s)=I.
\end{align*}
The normalization property is understood as the identity
\begin{align*}
 M(-s)^\top M(s)+N(-s)^\top N(s)=I.
\end{align*}
The relevance of such factorizations is that they relate to the graph of the transfer function $G(s)=N(s)M(s)^{-1}:$ 
\begin{align*}
 \mathrm{im} \begin{bmatrix} M \\ N \end{bmatrix} = \left\{ \begin{bmatrix} Mf \\ Nf \end{bmatrix} \ \colon \ f \in \mathcal{H}_2^{m, 1} \right\} \subset \mathcal{H}_2^{2m,1}.
\end{align*}
Moreover, for two transfer functions $G_1,G_2$ with normalized coprime factorizations $\begin{bsmallmatrix}M_1 \\ N_1 \end{bsmallmatrix}$ and $\begin{bsmallmatrix} M_2\\ N_2 \end{bsmallmatrix}$, the gap metric can be defined (see \cite{SefO93}) as follows 
\begin{align*}
 \delta_{\mathrm{gap}}(G_1,G_2) =\max \left\{ \vec{\delta}_{\mathrm{gap}}(G_1,G_2),\vec{\delta}_{\mathrm{gap}}(G_2,G_1) \right\},
\end{align*}
where the directed gap $\vec{\delta}_{\mathrm{gap}}(G_1,G_2)$ is given by 
\begin{align*}
 \vec{\delta}_{\mathrm{gap}}(G_1,G_2)=\inf\limits_{\Pi\in \mathcal{H}_\infty^{m,m}} \left\|\begin{bmatrix} M_1 \\ N_1 \end{bmatrix} - \begin{bmatrix} M_2 \\ N_2 \end{bmatrix}\Pi \right\|_{\mathcal{H}_\infty}.
\end{align*}
For classical LQG balanced truncation (i.e.~when $R=0$ in \eqref{eq:lqg_fare_ph}), we have an error bound of the form 
\begin{align}\label{eq:class_LQG_err_bnd}
 \vec{\delta}_{\mathrm{gap}}(G,\widehat{G}_r)\le  \left\|\begin{bmatrix} M \\ N \end{bmatrix} - \begin{bmatrix} \widehat{M}_r \\ \widehat{N}_r \end{bmatrix} \right\|_{\mathcal{H}_\infty}:=\sup_{z\in \mathbb C_{+}}\left\| \begin{bmatrix} M(z) \\ N(z) \end{bmatrix}- \begin{bmatrix} \widehat{M}_r(z) \\ \widehat{N}_r(z) \end{bmatrix}\right\|_2\le 2\sum_{i=r+1}^n \theta_i,
\end{align}
where $\theta_i=\frac{\sigma_i}{\sqrt{1+\sigma_i^2}}$ and $\sigma_i=\sqrt{\lambda_i(\Phatf\Phatc)}$ are the so-called LQG characteristic values and where $\begin{bsmallmatrix} \widehat{M}_r\\ \widehat{N}_r \end{bsmallmatrix}$ is a normalized right coprime factorization of the transfer function of the reduced-order model. 
Let us further emphasize that this gap metric error bound is particularly useful for analyzing the closed loop behavior of the associated reduced LQG controller. For more details on specific closed loop estimates, we refer to, e.g., \cite{Vid84}.

\subsection{Lyapunov inequalities for structure-preserving LQG balanced truncation}

Let us return to the general case where $R\succeq 0$ in \eqref{eq:lqg_fare_ph} and establish an error bound structurally identical to \eqref{eq:class_LQG_err_bnd}. For this purpose, in what follows we exploit an approach from \cite{DamB14} (similar to the arguments provided in \cite{SanR04}) where specific Lyapunov inequalities have been used to derive classical error bounds. 
Note that $\Phatc$ satisfies
\begin{align}
 A_{\Phatc}^\top \Phatc + \Phatc A_{\Phatc} + C^\top C +\Phatc BB^\top \Phatc =0.
\end{align}
In other words, $\Phatc$ coincides with the observability Gramian of the system
\begin{equation}\label{eq:coprime_real}
\begin{aligned}
    \dot{w} &= A_{\Phatc} w + B u, \quad w(0)=0, \\
    y&= \underbrace{\begin{bmatrix} -B^\top \Phatc \\ C \end{bmatrix}}_{C_{\Phatc}} w+ \begin{bmatrix}  I_m \\ 0 \end{bmatrix} u
\end{aligned}
\end{equation}
which is a realization of $\begin{bsmallmatrix} M\\N \end{bsmallmatrix}$. 
In the standard case, i.e.~using $\Qf=BB^\top$, the controllability Gramian of this system is given by $(I_n+\Phatf\Phatc)^{-1}\Phatf$ (see, e.g., \cite{Cur03}) and, thus, balancing the matrices $\Phatf$ and $\Phatc$ will also transform the controllability and observability Gramians of $\begin{bsmallmatrix}M\\ N \end{bsmallmatrix}$ into diagonal (though not necessarily equal) form.
The following proposition shows that in our setting, i.e.~using $\Qf=BB^\top+2R$, the matrix $(I_n+\Phatf \Phatc)^{-1}\Phatf$ is in general no longer the controllability Gramian of \eqref{eq:coprime_real}, but it still satisfies at least an associated Lyapunov inequality.

\begin{proposition}
    Let $(J,R,Q,B)$ define a minimal port-Hamiltonian system of the form \eqref{eq:pH} and let $\Phatc$ and $\Phatf$ denote the unique stabilizing solutions to the control and filter Riccati equations \eqref{eq:lqg_care_ph} and \eqref{eq:lqg_fare_ph}, respectively. 
    For the system \eqref{eq:coprime_real} and $\mathcal{L}=(I_n+\Phatf \Phatc)^{-1}\Phatf$ it holds that 
  \begin{align}\label{eq:obs_ineq_cp}
  A_{\Phatc} \mathcal{L} + \mathcal{L} A_{\Phatc}^\top+BB^\top \preceq 0.
  \end{align}
\end{proposition}

\begin{proof}
  \newcommand{\apf}{A_{\Phatf}}
  \newcommand{\tpf}{\Phatf}
  \newcommand{\tpc}{\Phatc}
  \newcommand{\tpci}{\tpc^{-1}}
  Using $\Phatf=\Phatf^\top$ and $\Phatc=\Phatc^\top$, one easily verifies that $\mathcal{L}=\mathcal{L}^\top$. Hence, instead of \eqref{eq:obs_ineq_cp} we may show that 
  \begin{align*}
   A_{\Phatc} \mathcal{L}^\top + \mathcal{L} A_{\Phatc}^\top+BB^\top \preceq 0.
  \end{align*}
  Since $I_n+\Phatf \Phatc$ is invertible, this is however equivalent to showing that 
 \begin{align*}
 (I_n+\Phatf \Phatc) A_{\Phatc} \Phatf 
 +\Phatf A_{\Phatc}^\top  (I_n+\Phatc\Phatf)
 + (I_n+\Phatf \Phatc) BB^\top (I_n+\Phatc \Phatf) \preceq 0.
 \end{align*}
 We now obtain
 \begin{align*}
     & (I_n+\Phatf \Phatc) A_{\Phatc} \Phatf
 +\Phatf A_{\Phatc}^\top  (I_n +\Phatc \Phatf)
 + (I_n+\Phatf\Phatc) BB^\top (I_n + \Phatc \Phatf ) \\
 &= (I_n+\Phatf \Phatc) A \Phatf + \Phatf A^\top (I_n+\Phatc \Phatf) \\ 
&\quad - (I_n+\Phatf \Phatc)BB^\top \Phatc\Phatf - \Phatf \Phatc BB^\top (I_n+\Phatc \Phatf)  \\
 &\quad + BB^\top + \Phatf \Phatc BB^\top+ BB^\top \Phatc \Phatf + \Phatf \Phatc BB^\top \Phatc \Phatf \\
  &= A \Phatf + \Phatf A^\top  + BB^\top  \\
  &\quad - \Phatf \Phatc BB^\top \Phatc \Phatf+\Phatf \Phatc A \Phatf+ \Phatf A^\top \Phatc \Phatf \\
  &= -2R + \Phatf\left( C^\top C + \Phatc A  +A^\top  \Phatc -  \Phatc^\top BB^\top  \Phatc  \right) \Phatf= -2R \preceq 0 . \qedhere
 \end{align*}
\end{proof}

We are ready to state an error bound for a reduced-order model obtained either by Algorithm \ref{alg:new} or (equivalently) by balanced truncation w.r.t.\@ the Gramians $\Pc$ and $\Pf$ from Appendix \ref{apdxA}.

\begin{theorem}\label{thm:err_bnd_lqg}
    Let $(J,R,Q,B)$ define a minimal port-Hamiltonian system of the form \eqref{eq:pH} which is balanced w.r.t.~the stabilizing solutions of \eqref{eq:lqg_care_ph} and \eqref{eq:lqg_fare_ph}.
    Furthermore, let $(\widehat{A}_r,\widehat{B}_r,\widehat{C}_r)$ be a reduced-order model obtained by Algorithm \ref{alg:new}. 
    Denote the balanced Gramians $\Phatf=\Phatc=\Sigma=\mathrm{diag}(\Sigma_1,\Sigma_2)$, where $\Sigma_2=\mathrm{diag}(\sigma_{r+1},\dots,\sigma_n).$ 
    If $\begin{bsmallmatrix} M\\ N \end{bsmallmatrix}$ and $\begin{bsmallmatrix} \widehat{M}_r\\ \widehat{N}_r \end{bsmallmatrix}$ are the associated right coprime factorizations as specified in \eqref{eq:coprime_real}, then 
    \begin{align}\label{eq:err_bnd_lqg}
        \left\| \begin{bmatrix} M \\ N \end{bmatrix}  - \begin{bmatrix} \widehat{M}_r \\ \widehat{N}_r \end{bmatrix}\right\|_{\mathcal{H}_\infty} \le 2  \sum_{i=r+1}^n \frac{\sigma_i}{\sqrt{1+\sigma_i^2}}.
    \end{align}
\end{theorem}

\begin{proof}
 The proof uses almost the exact same arguments as given in \cite{DamB14}. 
 In contrast to the latter reference, here the matrices $\Phatc$ and $\mathcal{L}$ satisfying the Lyapunov (in)equalities do not coincide.
 While this only requires minor modifications in the reasoning from \cite{DamB14}, for self-consistency we provide a full proof of the assertion.
 
 First note that since $(A,B,C)$ is balanced, from the above considerations, it follows that 
 \begin{align*}
     \mathcal{L}=(I_n+\Phatf \Phatc)^{-1}\Phatf =  (I_n + \Sigma^2)^{-1}\Sigma = \mathrm{diag}\left(\frac{\sigma_i}{1+\sigma_i^2}\right), \ i=1,\dots,n,
 \end{align*}
 as well as 
 \begin{align}\label{eq:orig_lqg_bal} 
A_{\Phatc}^\top \Sigma + \Sigma A_{\Phatc} + C_{\Phatc}^\top C_{\Phatc} = 0,\quad  
 A_{\Phatc}  \mathcal{L} + \mathcal{L} A_{\Phatc}^\top +BB^\top  \preceq 0.
\end{align}
Similarly, for the reduced-order surrogate of \eqref{eq:coprime_real} which is defined by 
\begin{align*}
\widehat{A}_{\Sigma}&=\widehat{A}_r- \widehat{B}_r\widehat{B}_r^\top\Sigma_1, \quad \widehat{B}_{\Sigma}=\widehat{B}_r, \quad \widehat{C}_{\Sigma}= \begin{bmatrix} -\widehat{B}_r^\top \Sigma_1 \\ \widehat{C}_r \end{bmatrix}
\end{align*}
 we conclude that
\begin{equation}\label{eq:red_lqg_bal}
\begin{aligned}
\widehat{A}_{\Sigma}^\top \Sigma_1 +\Sigma_1 \widehat{A}_{\Sigma}+\widehat{C}_{\Sigma}^\top \widehat{C}_{\Sigma}  = 0,\quad 
 \widehat{A}_{\Sigma} \mathcal{L}_1 + \mathcal{L}_1 \widehat{A}_{\Sigma}^\top + \widehat{B}_{\Sigma} \widehat{B}_{\Sigma}^\top \preceq 0 .
 \end{aligned}
\end{equation}
where $\mathcal{L}_1=\mathrm{diag}(\frac{\sigma_i}{1+\sigma_i^2}), i=1,\dots,r.$
 
 Since the $\mathcal{H}_\infty$-norm is the $L_2$-$L_2$-induced norm between inputs and outputs (e.g., \cite[Section 5]{Ant05a}), we thus focus on the two systems 
 
 \begin{align*}
 \dot{w}&=A_{\Phatc} w + B u, \quad w(0)=0, \quad  y= C_{\Phatc}w+\begin{bmatrix}I_m \\ 0 \end{bmatrix} u, \\
  \dot{\widehat{w}}_r&=\widehat{A}_{\Sigma} \widehat{w}_r + \widehat{B}_{\Sigma} u, \quad \widehat{w}_r(0)=0,\quad   \widehat{y}_r= \widehat{C}_{\Sigma}\widehat{w}_r+\begin{bmatrix}I_m \\ 0 \end{bmatrix}u.
\end{align*}
Since by assumption $(A,B,C)$ is balanced, i.e., $\Phatf=\Sigma=\Phatc$ we also obtain
\begin{align*}
A_{\Phatc}&=A-BB^\top\Sigma =\begin{bmatrix} A_{11} & A_{12} \\ A_{21} & A_{22} \end{bmatrix}:= \begin{bmatrix} \widehat{A}_r & * \\ * & * \end{bmatrix} - \begin{bmatrix} \widehat{B}_r \\ B_2 \end{bmatrix} \begin{bmatrix} \widehat{B}_r \\ B_2 \end{bmatrix}^\top \begin{bmatrix} \Sigma_1 & 0 \\ 0 & \Sigma_2 \end{bmatrix}  = \begin{bmatrix} \widehat{A}_r-\widehat{B}_r \widehat{B}_r^\top \Sigma_1 & * \\ * & * \end{bmatrix} \\
C_{\Phatc} &:= \begin{bmatrix} C_1 & C_2\end{bmatrix}= \begin{bmatrix} -\widehat{B}_r^\top \Sigma_1  & * \\ \widehat{C}_r & *  \end{bmatrix}.
\end{align*}
In other words, the reduced right coprime factorization $(\widehat{A}_{\Sigma},\widehat{B}_{\Sigma},\widehat{C}_{\Sigma})$ is obtained by truncating the original right coprime factorization $(A_{\Phatc},B,C_{\Phatc})$. 
With the state vector $w=\begin{bsmallmatrix} w_1 \\ w_2 \end{bsmallmatrix}$ being partitioned accordingly, we obtain 
\begin{align*}
 \frac{\dd}{\dd t} \langle w_1 - \widehat{w}_r,\Sigma_1(w_1 -\widehat{w}_r)\rangle &= 2 \langle w_1-\widehat{w}_r,\Sigma_1 [A_{11},A_{12}] \begin{bmatrix} w_1-\widehat{w}_r \\ w_2 \end{bmatrix} \rangle \\
 \frac{\dd}{\dd t} \langle w_2 ,\Sigma_2 w_2\rangle &= 2\langle w_2,\Sigma_2 (A_{21}w_1 + A_{22}w_2 + B_2 u)\rangle.
\end{align*}
Note that it holds that $y-\widehat{y}_r=C_{\Phatc} \begin{bsmallmatrix} w_1-\widehat{w}_r \\ w_2 \end{bsmallmatrix}$ which we use to show
\begin{align*}
 -\| y-\widehat{y}_r\|^2&=-\begin{bmatrix} w_1-\widehat{w}_r \\ w_2 \end{bmatrix}^\top C_{\Phatc}^\top C_{\Phatc} \begin{bmatrix} w_1-\widehat{w}_r \\ w_2 \end{bmatrix} = \begin{bmatrix} w_1-\widehat{w}_r \\ w_2 \end{bmatrix}^\top (A_{\Phatc}^\top \Sigma +\Sigma A_{\Phatc} ) \begin{bmatrix} w_1-\widehat{w}_r \\ w_2 \end{bmatrix} \\[1ex]
 &=2 (w_1-\widehat{w}_r)^\top  \Sigma_1 \begin{bmatrix}A_{11} & A_{12} \end{bmatrix} \begin{bmatrix} w_1-\widehat{w}_r \\ w_2 \end{bmatrix}
 +2 w_2 ^\top \Sigma_2 \begin{bmatrix} A_{21} & A_{22} \end{bmatrix} \begin{bmatrix} w_1 - \widehat{w}_r \\ w_2 \end{bmatrix}.
\end{align*}
Combining all of the previous results leads to
\begin{align*}
 -\| y-\widehat{y}_r\|^2 \frac{\dd}{\dd t} \langle w_1 - \widehat{w}_r,\Sigma_1(w_1 -\widehat{w}_r)\rangle + \frac{\dd}{\dd t} \langle w_2,\Sigma_2 w_2\rangle-2w_2^\top \Sigma_2 (A_{21}\widehat{w}_r+B_2u)
\end{align*}
and finally
\begin{align}\label{eq:aux2}
 \int_0^T \|y(t)-\widehat{y}_r(t)\|^2 \dd t \le 2\int_0^T w_2(t)^\top \Sigma_2 (A_{21}\widehat{w}_r(t)+B_2u(t))\dd t.
\end{align}
Still following \cite{DamB14}, we want to show that 
\begin{align}\label{eq:aux3}
 4\int_0^T \| u(t)\| ^2 \dd t \ge 2 \int_0^T w_2(t)^\top \mathcal{L}_2^{-1}(A_{21}\widehat{w}_r(t)+B_2 u(t))\dd t),
\end{align}
where $\mathcal{L}_2=\mathrm{diag}(\frac{\sigma_i}{1+\sigma_i^2}),i=r+1,\dots,n.$
From \eqref{eq:orig_lqg_bal} it follows that 
\begin{align*}
    \mathcal{L}^{-1} A_{\Phatc}+ A_{\Phatc}^\top \mathcal{L}^{-1} + \mathcal{L}^{-1} BB^\top \mathcal{L}^{-1} \preceq 0
\end{align*}
which due to Schur complement properties also implies that
\begin{align*}
  \begin{bmatrix}  \mathcal{L}^{-1} A_{\Phatc}+ A_{\Phatc}^\top \mathcal{L}^{-1} & \mathcal{L}^{-1}B \\ B^\top \mathcal{L}^{-1} & -I_m \end{bmatrix} \preceq 0.
\end{align*}
This can be rewritten as follows
\begin{align*}
 \begin{bmatrix} 0 & 0 \\ 0 & I_m \end{bmatrix}\succeq \begin{bmatrix}A_{\Phatc} & B \\ I_n & 0 \end{bmatrix}^\top \begin{bmatrix} 0 & \mathcal{L}^{-1} \\ \mathcal{L}^{-1} & 0 \end{bmatrix} \begin{bmatrix} A_{\Phatc} & B \\ I_n & 0 \end{bmatrix}.
\end{align*}
Let us multiply the previous inequality with $\begin{bmatrix} (w_1 + \widehat{w}_r)^\top & w_2^\top & 2 u^\top\end{bmatrix}$ and $\begin{bsmallmatrix} w_1 + \widehat{w}_r \\ w_2 \\ 2 u \end{bsmallmatrix}$ to obtain 
\begin{align*}
 4\| u \|^2 &\ge 2(w_1+\widehat{w}_r)^\top\mathcal{L}_1^{-1}(\begin{bmatrix} A_{11} & A_{12} \end{bmatrix} \begin{bmatrix} w_1 + \widehat{w}_r\\ w_2 \end{bmatrix}+2B_1u) \\
 &\qquad + 2w_2^\top \mathcal{L}_2^{-1}(\begin{bmatrix} A_{21} & A_{22} \end{bmatrix} \begin{bmatrix} w_1 + \widehat{w}_r \\ w_2 \end{bmatrix}+2B_2 u ).
 \end{align*}
As before, we obtain 
\begin{align*}
 \frac{\dd}{\dd t} (w_1+\widehat{w}_r)^\top \mathcal{L}_1^{-1} (w_1 + \widehat{w}_r) &= 2 (w_1+\widehat{w}_r)^\top \mathcal{L}_1^{-1} (\begin{bmatrix} A_{11} & A_{12} \end{bmatrix} \begin{bmatrix} w_1 + \widehat{w}_r \\ w_2 \end{bmatrix} + 2B_1 u ) \\
 \frac{\dd}{\dd t} (w_2^\top \mathcal{L}_2^{-1} w_2) &= 2 w_2^\top \mathcal{L}_2^{-1}(A_{21}w_1 + A_{22}w_2 + B_2 u).
\end{align*}
Combining the last three (in)equalities leads to 
\begin{align*}
 4\|u\|^2 \ge \frac{\dd}{\dd t} ((w_1+\widehat{w}_r)^\top \mathcal{L}_1^{-1} (w_1 +\widehat{w}_r))+\frac{\dd}{\dd t}(w_2^\top \mathcal{L}_2^{-1}w_2)+2 w_2^\top \mathcal{L}_2^{-1}(A_{21}\widehat{w}_r+B_2 u).
\end{align*}
Integration and using $w(0)=0$, $\widehat{w}_r(0)=0$ yields \eqref{eq:aux3}.

If $\Sigma_2=\sigma_{\nu}$, then $\mathcal{L}_2=\frac{\sigma_{\nu}}{1+\sigma_{\nu}^2}$. Multiplication of \eqref{eq:aux3} with $\frac{\sigma_{\nu}^2}{1+\sigma_{\nu}^2}$ implies 
\begin{align*}
 \frac{4\sigma_{\nu}^2}{1+\sigma_{\nu}^2}  \int_0^T \|u(t)\|^2 \dd t &\ge 2 \int_0^T w_2(t)^\top \frac{\sigma_{\nu}^2}{1+\sigma_{\nu}^2}\mathcal{L}_2^{-1}(A_{21}\widehat{w}_r(t)+B_2u(t))\dd t) \\
 &= 2\int_0^T w_2(t)^\top \Sigma_2 (A_{21}\widehat{w}_r(t)+B_2u(t))\dd t) \ge \int_0^T \| y(t)-\widehat{y}_r(t)\|^2\dd t.
\end{align*}
We can now proceed recursively and truncate step by step one after the other state such that finally we obtain 
\begin{align*}
 \|y-\widehat{y}_r\|_{L^2([0,T];\mathbb R^m)} \le 2 (\theta_{r+1}+\cdots+\theta_{n})\| u\|_{L^2([0,T];\mathbb R^m)},
\end{align*}
where $\theta_i=\frac{\sigma_i}{\sqrt{1+\sigma_i^2}}$, for $i=r+1,\dots,n$. 
As the previous bound is independent of the time $T$, we can consider its limit $T\to \infty$ which shows the assertion.
\end{proof}

\begin{remark}
Note that the final arguments remain valid if $\Sigma_2=\sigma_\nu I$, i.e., in the case where some of the characteristic values appear multiple times. In this sense, the original formulation from \cite{DamB14} is beneficial since the summation in the error bound only includes distinct characteristic values. For $R=0$, the two Gramians coincide with the ones used in classical LQG balanced truncation. While due to the special pH structure  this would result in a purely Hamiltonian system, the right coprime framework could still be followed as long as $(J,B,C)$ is minimal.
\end{remark}

\begin{remark}
    The error bound in Theorem~\ref{thm:err_bnd_lqg} scales with the $\sigma_i$ values which are the square roots of the eigenvalues of $\Phatf\Phatc$.
    In particular, we observe that these values are invariant under similarity transformations, which follows from the fact that any similarity transformation based on an invertible matrix $T$ leads to the transformed Gramians
    $\overline{\mathcal{P}}_{\mathrm{f}} = T\Phatf T^\top$ and $\overline{\mathcal{P}}_{\mathrm{c}} = T^{-\top}\Phatc T^{-1}$ which satisfy
    \begin{equation*}
        \overline{\mathcal{P}}_{\mathrm{f}} \overline{\mathcal{P}}_{\mathrm{c}} = T\Phatf \Phatc T^{-1}.
    \end{equation*}
    Thus, $\overline{\mathcal{P}}_{\mathrm{f}} \overline{\mathcal{P}}_{\mathrm{c}}$ and $\Phatf \Phatc$ have the same eigenvalues and, hence, the $\sigma_i$ values are invariant under similarity transformations.
    For the special case $R=0$, these coincide with the classical LQG characteristic values, cf.~subsection~\ref{sec:standardLQGBT}.
\end{remark}

\subsection[Minimizing the error bound - choosing the right Q]{Minimizing the error bound - choosing the right $Q$}\label{subsec:choosing_Q}

While the error bound in Theorem \ref{thm:err_bnd_lqg} is structurally the same as in classical LQG balanced truncation, the characteristic values are derived from a filter Riccati equation including the additional term $2R.$ Since we have seen that this specific choice of the covariance matrix implies $\Phatf=Q^{-1}$, the question arises whether there exists a particularly \emph{good} pH representation in the sense that the error bound provided in Theorem \ref{thm:err_bnd_lqg} is minimized. With this in mind, first note that the error bound is determined by the eigenvalues of the matrix 
\begin{align*}
 \mathcal{L} \Phatc = (I_n+\Phatf \Phatc)^{-1}\Phatf \Phatc= (I_n+Q^{-1}\Phatc)^{-1}Q^{-1} \Phatc= (Q+\Phatc)^{-1} \Phatc.
\end{align*}
Since $\Phatc=\Phatc^\top \succ 0$, we may consider a Cholesky decomposition $\Phatc=L_{\Phatc}^\top L _{\Phatc} $ such that:
\begin{align}\label{eq:lgq_char_val}
\theta_i=\sqrt{
\lambda_i( \mathcal{L}\Phatc)}=\sqrt{\lambda_i(L_{\Phatc} \mathcal{L} \Phatc  L_{\Phatc}^{-1} )}=\sqrt{\lambda_i(L_{\Phatc} (Q+\Phatc)^{-1}L_{\Phatc}^\top)}.
\end{align}
Following \cite{BeaMV19}, let us assume that $X=X^\top \succ 0$ is a solution to  \eqref{eq:kyp-lmi} and consider the associated alternative pH representation \eqref{eq:pHX}.
Since \eqref{eq:lqg_care_ph} only depends on $(A,B,C)$ its solution is independent of $R_X$ and still given by $\Phatc$. On the other hand, \eqref{eq:lqg_fare_ph} explicitly depends on $R_X,$ which for $X\neq Q$ will generally be affected by the previous transformation. In particular, for $\theta_i$, we have achieved the following transformation
\begin{align*}
    \theta_i = \sqrt{\lambda_i(L_{\Phatc} (Q+\Phatc)^{-1}L_{\Phatc}^\top)} \quad \to \quad \sqrt{\lambda_i(L_{\Phatc} (X+\Phatc)^{-1}L_{\Phatc}^\top))} = \widehat{\theta}_i.
\end{align*}
Assume now that $X_{\mathrm{max}}=X_{\mathrm{max}}^\top \succ 0$ is the \emph{maximal} solution to \eqref{eq:kyp-lmi}. Hence, it holds that $X_{\mathrm{max}}\succeq Q$ as well as
\begin{align*}
    X_{\mathrm{max}}+\Phatc\succeq Q+\Phatc\succ 0
\end{align*}
also implying that
\begin{align*}
0\prec ( X_{\mathrm{max}}+\Phatc)^{-1} \preceq (Q+\Phatc)^{-1}.
\end{align*}
Using the Courant-Fischer-Weyl min-max principle (\cite[Theorem 8.1.2]{GolV13}) allows us to conclude that
\begin{align*}
    \theta_i^2 &= \lambda_i(L_{\Phatc} (Q+\Phatc)^{-1}L_{\Phatc}^\top))\\ 
    &=\min_{\substack{\mathcal{X}_i\subset \mathbb R^n\\ \mathrm{dim}(\mathcal{X}_i)=i}}\ \max_{\substack{z\in \mathcal{X}_i\\ \|z\|=1}} z^\top (L_{\Phatc} (Q+\Phatc)^{-1}L_{\Phatc}^\top))z \\ & \ge \min_{\substack{\mathcal{X}_i\subset \mathbb R^n\\ \mathrm{dim}(\mathcal{X}_i)=i}}\ \max_{\substack{z\in \mathcal{X}_i\\ \|z\|=1}} z^\top (L_{\Phatc} (X+\Phatc)^{-1}L_{\Phatc}^\top))z \\
    &=\lambda_i(L_{\Phatc} (X+\Phatc)^{-1}L_{\Phatc}^\top))=\widehat{\theta}_i^2.
\end{align*}
The previous considerations suggest to first replace \eqref{eq:pH} with the alternative representation \eqref{eq:pHX} by computing the maximal solution $X_{\mathrm{max}}$ to \eqref{eq:kyp-lmi} in order to minimize the error bound in Theorem \ref{thm:err_bnd_lqg}. Let us emphasize that such a computation is numerically demanding and the results first of all are of theoretical nature.
Furthermore, it should be noted that the transformation corresponds to choosing a different Hamiltonian, that might in turn have a different physical meaning. Whether this affects positively or negatively the method depends on the specific application.%

\begin{remark}
    \label{rem:mixedBalancing}
    In this subsection, we demonstrated that exchanging the weighting matrix $BB^\top+2R$ by $BB^\top+2R_{X_{\mathrm{max}}}$ yields the Gramian $\Phatf = X_{\mathrm{max}}^{-1}$ instead of $\Phatf = Q^{-1}$.
    This approach allows for an alternative interpretation of the proposed balancing procedure.
    To this end, we point out that the Gramian $\Phatf = X_{\mathrm{max}}^{-1}$ is not only the unique stabilizing solution of the modified filter algebraic Riccati equation \eqref{eq:lqg_fare_ph} with weighting matrix $BB^\top+2R_{X_{\mathrm{max}}}$, but it is also the minimal solution of the dual KYP-LMI \eqref{eq:dualKYPLMI}.
    Thus, the proposed balancing procedure can be regarded as a mixture between classical LQG balanced truncation and positive real balanced truncation, in the sense that we balance the stabilizing solution of the algebraic Riccati equation \eqref{eq:lqg_care_ph} and the minimal solution of the dual KYP-LMI \eqref{eq:dualKYPLMI}.
    A similar mixed approach has been proposed in \cite{UnnVE07}, where the solution of a Lyapunov equation and the solution of an algebraic Riccati equation have been balanced.
\end{remark}

\section{Consequences for standard balanced truncation}\label{sec:classic_bt}

In this section, we briefly discuss how the above framework can be used to modify standard balanced truncation for port-Hamiltonian systems such that analogous results and error bounds hold true. 
Note that for balanced truncation realized within the effort-constraint reduction framework (see \cite{PolvdS12}), in \cite{Wuetal18} the authors have derived an error bound based on an auxiliary system. 
Our approach is different in the sense that in specific situations, a classical balanced truncation error bound can be shown to hold true.
As is common in the context of balanced truncation, we assume in this section that $A=(J-R)Q$ is asymptotically stable, i.e., all eigenvalues of $A$ lie in the open left half plane.

The main idea is to transfer the previous concepts to the setting of standard balanced truncation. 
In particular, we propose to replace the constant term $BB^\top$ by $2R$ in the standard controllability Lyapunov equation and, thus, to balance the system with respect to the solutions of the (modified) controllability and observability Lyapunov equations 
\begin{equation}\label{eq:lyap}
\begin{aligned}
A \Lcalc+ \Lcalc A^\top + 2R &=0, \\
A^\top \Mcalo + \Mcalo A + C^\top C &=0. 
\end{aligned}
\end{equation} 

With this specific choice of the constant term $2R$ in the first equation, one immediately verifies that $\Lcalc=Q^{-1}$ for a pH system \eqref{eq:pH}.
Consequently, in balanced coordinates, it holds that 
\begin{align*}
   \Ab \Pi +\Pi  \Ab^\top + 2 \Rb&=0, \\
  \Ab^\top \Pi + \Pi \Ab + \Cb^\top \Cb  &= 0,
\end{align*}
where $\Pi=\mathrm{diag}(\pi_1,\dots,\pi_n)$. In particular, we have $\Qb=\Pi^{-1}$ such that an effort-constraint reduced-order model satisfies 
\begin{align*}
    Q_r=\QbUpperLeft-\QbUpperRight\QbLowerRight^{-1} \QbLowerLeft=\QbUpperLeft=\mathrm{diag}\left(\frac1{\pi_1},\dots,\frac1{\pi_r}\right).
\end{align*}
Note that an effort-constraint reduced-order model automatically is port-Hamiltonian and is further obtained by simple truncation of a system balanced w.r.t.\@ modified system Gramians. With regard to an error bound as in Theorem \ref{thm:err_bnd_lqg}, we make an additional assumption on $R$ and $B.$ We summarize our findings in the following result. 
\begin{corollary}\label{cor:bt}
    Let $(J,R,Q,B)$ with asymptotically stable $A=(J-R)Q$ define a minimal port-Hamiltonian system of the form \eqref{eq:pH} which is balanced w.r.t.~the solutions of \eqref{eq:lyap}.
    Furthermore, let $(\widehat{A}_r,\widehat{B}_r,\widehat{C}_r)$ be a reduced-order model obtained by truncation of $(A,B,C)$ and let $c>0$ be such that $c R\succeq  \frac{1}{2}BB^\top$.
    Besides, the balanced Gramians are denoted by $\Mcalo=\Lcalc=\Pi=\mathrm{diag}(\Pi_1,\Pi_2),$ where $\Pi_2=\mathrm{diag}(\pi_{r+1},\dots,\pi_n).$ 
    If $G(\cdot)=C(\,\cdot\, I_n-A)^{-1}B$ and $\widehat{G}_r(\cdot)=\widehat{C}_r (\,\cdot\, I_r - \widehat{A}_r)^{-1} \widehat{B}_r$ are the associated transfer functions, then 
    \begin{align}\label{eq:Hinf_err_bnd}
        \|G-\widehat{G}_r\|_{\mathcal{H}_\infty} \le 2\sqrt{c} \sum_{i=r+1}^n  \pi_i.
    \end{align}
    Moreover, the reduced-order model $(\widehat{A}_r,\widehat{B}_r,\widehat{C}_r)$ is port-Hamiltonian.
\end{corollary}

\begin{proof} 
Note that the assumption on $R$ and $B$ implies that for $\Lambda =c\Pi,$ we have
\begin{align*}
  0=c (A\Pi + \Pi A^\top + 2R) = A\Lambda + \Lambda A^\top +2 cR\succeq A\Lambda + \Lambda A^\top + B B^\top.
\end{align*}
The proof then follows along the lines of the proof of Theorem \ref{thm:err_bnd_lqg} with $\mathcal{L}, \Phatf $ being replaced by $\Lcalc,\Mcalo,$ and is thus omitted here.
\end{proof}

We proceed with a discussion on the condition $cR\succeq\frac{1}{2}BB^\top$ that appears in Corollary \ref{cor:bt}.
First of all, we investigate necessary and sufficient conditions for the existence of such a $c>0$.

\begin{lemma}\label{lem:cRB_cond}
    Let $R\in\mathbb{R}^{n\times n}$ with $R=R^\top\succeq0$ and $B\in\mathbb{R}^{n \times m}$ be two matrices.
    Then there exists a constant $c>0$ such that $cR\succeq \frac{1}{2}BB^\top$ if and only if $\operatorname{Im}(B)\subseteq\ker(R)^\perp$ or, equivalently, $\operatorname{Im}(B)\subseteq\operatorname{Im}(R)$.
\end{lemma}
\begin{proof}
    Since $R$ is symmetric, it is well-known that $\ker(R)^\perp=\operatorname{Im}(R)$. Therefore it is sufficient to prove the first statement.
    Suppose first that there exists such a $c>0$. Let $z_1=Bu$ for some $u\in\mathbb{R}^m$ be a generic element of $\operatorname{Im}(B)$,
    and let $z_2\in\ker(R)$: we need to show that $z_1^\top z_2=0$.
    Since
    \begin{equation*}
        0 = z_2^\top R z_2 \geq \frac{1}{2c}z_2^\top BB^\top z_2 = \frac{1}{2c}\Vert B^\top z_2\Vert_2^2
        \quad\Rightarrow\quad B^\top z_2=0,
    \end{equation*}
    we deduce that $z_1^\top z_2=u^\top B^\top z_2=0$, thus $\operatorname{Im}(B)\subseteq\ker(R)^\perp$.
    
    Suppose now that $\operatorname{Im}(B)\subseteq\ker(R)^\perp$. 
    Note that $\ker(B^\top)=\ker(BB^\top)$, since $z^\top BB^\top z=\Vert B^\top z\Vert_2^2$ for all $z\in\mathbb{R}^n$.
    In particular, $\operatorname{Im}(BB^\top)=\ker(BB^\top)^\perp=\ker(B^\top)^\perp=\operatorname{Im}(B)\subseteq\operatorname{Im}(R)$.
    Let
    \begin{equation*}
        U^\top R U = \bmat{\Lambda & 0 \\ 0 & 0}, \quad
        0\prec \Lambda \in\mathbb{R}^{k\times k} \text{ diagonal}, \quad
        U^\top U=I_n
    \end{equation*}
    be the SVD of $R$, and let $\lambda_m>0$ be the minimum diagonal entry of $\Lambda$.
    Since $\operatorname{Im}(BB^\top)\subseteq\operatorname{Im}(R)$, we clearly have
    \begin{equation*}
        U^\top BB^\top U = \bmat{S & 0 \\ 0 & 0}, \quad
        S = S^\top \in \mathbb{R}^{k\times k}.
    \end{equation*}
    In particular, $\Vert S\Vert_2=\Vert BB^\top\Vert_2=\Vert B\Vert_2^{2}$, and for any $z=(z_1,z_2)\in\mathbb{R}^{n}=\mathbb{R}^{k+(n-k)}$ we have
    \begin{align*}
        z^\top U^\top\frac{1}{2}BB^\top Uz
        &= \frac{1}{2}z_1^\top Sz_1 \leq \frac{1}{2}\Vert S\Vert_2 z_1^\top z_1 \leq \frac{\Vert B\Vert_2^{2}}{2\lambda_m}z_1^\top\Lambda z_1 = \frac{\Vert B\Vert_2^{2}}{2\lambda_m}z^\top U^\top RUz.
    \end{align*}
    Since $U$ is invertible, this is equivalent to $cR\succeq\frac{1}{2}BB^\top$, with $c=\frac{\Vert B\Vert_2^{2}}{2\lambda_m}$.
\end{proof}

In other words, Lemma \ref{lem:cRB_cond} states that the condition of Corollary \ref{cor:bt} is satisfied exactly when the input port (represented by $B$) interacts directly with the dissipation port (represented by $R$). In particular, intuition tells us that, if the condition holds, any small input will be damped by the dissipation of the system and will preserve the pH structure, while if the condition fails, there are small inputs that will insert new energy in the system, without that the dissipation can immediately act on that. We formalize this intuition with the following result.

\begin{proposition}\label{prop:cRB_cond}
  Let a port-Hamiltonian system \eqref{eq:pH} be given by $(J,R,Q,B)$, and consider output feedbacks of the form $u=Fy=FB^\top Qx$ with a feedback matrix $F\in\mathbb{R}^{m\times m}$, leading to the closed loop system
  \begin{equation}\label{eq:closed_loop}
      \dot x = (J-R+BFB^\top)Qx.
  \end{equation}
  Then a constant $c>0$ satisfies the condition $cR\succeq \frac{1}{2}BB^\top$ if and only if \eqref{eq:closed_loop} is port-Hamiltonian with respect to $Q$ for all feedback matrices $F$ with $\Vert F\Vert_2\leq\frac{1}{2c}$.
\end{proposition}
\begin{proof}
    Let $\alpha=\frac{1}{2c}>0$, so that the two conditions are $R\succeq\alpha BB^\top$ and $\Vert F\Vert_2\leq\alpha$.
    Let $F$ be a generic feedback matrix with $\Vert F\Vert_2\leq\alpha$. If we split $F=\Fs+\Fk$ into its symmetric part $\Fs=\frac{1}{2}(F+F^\top)$ and skew-symmetric part $\Fk=\frac{1}{2}(F-F^\top)$, we can rewrite \eqref{eq:closed_loop} as
    \begin{equation*}
        \dot x = \big((J+B\Fk B^\top)-(R-B\Fs B^\top)\big)Qx = (J_F - R_F)Qx.
    \end{equation*}
    It is clear that this system is pH with respect to $Q$ if and only if $R_F=R_F^\top\succeq 0$.
    
    Suppose first that $R\nsucceq\alpha BB^\top$: then the feedback matrix $F=\Fs=\alpha I_m$ satisfies $\Vert F\Vert_2=\alpha$ but produces a closed loop system that is not pH with respect to $Q$, since $R_F=R-\alpha BB^\top\nsucceq 0$.
    Suppose now that $R\succeq\alpha BB^\top$. Note that $\Vert \Fs\Vert_2=\frac{1}{2}\Vert F+F^\top\Vert_2\leq\Vert F\Vert_2\leq\alpha$, therefore we can assume without loss of generality that $F=F^\top=\Fs$.
    Since $F$ is symmetric, there exist a matrix $G=G^\top\succeq 0$ such that $F\preceq G$ and $\Vert G\Vert_2=\Vert F\Vert_2\leq\alpha$, that can be constructed by taking the spectral decomposition of $F$ and replacing all negative eigenvalues with their absolute value, and a matrix $S=S^\top\succeq 0$, such that $G=S^2$ and $\Vert S\Vert_2^2=\Vert G\Vert_2\leq\alpha$, that can be constructed by replacing the eigenvalues of $G$ with their square root.
    In particular, we have
    \begin{align*}
        z^\top BFB^\top z &\leq z^\top BGB^\top z = z^\top BS^2B^\top z = \Vert SB^{\top}z\Vert_2^2 \leq \Vert S\Vert_2^2\Vert B^{\top}z\Vert_2^{2} \leq \alpha z^\top BB^\top z \leq z^\top Rz,
    \end{align*}
    i.e., $R_F=R-BFB^\top\succeq 0$, as requested.
\end{proof}

It should be stressed that Proposition \ref{prop:cRB_cond} does not provide necessary conditions for the port-Hamiltonian system to admit structure-preserving output feedback.
In fact, any output feedback of the form $u=Fy$ with $F+F^\top\preceq 0$ will lead to a closed loop port-Hamiltonian system of the form $\dot x=(\tilde J-\tilde R)Qx$, with $\tilde J=J+\frac{1}{2}B(F-F^\top)B^\top$ and $\tilde R=R-\frac{1}{2}B(F+F^\top)B^\top$ satisfying $\tilde J=-\tilde J^\top$ and $\tilde R=\tilde R^\top\succeq 0$, regardless of the form of $R$ and $B$, and of the magnitude of $\lVert F\rVert_2$.

In fact, Proposition \ref{prop:cRB_cond} provides conditions for the existence of output feedback with bounded norm that \emph{destroys} the port-Hamiltonian structure.
This can for example be useful in one of the following scenarios: if we expect disturbances in the output feedback, in which case $c$ can provide a measure of robustness, or if we are actually interested in destabilizing the system.

\begin{remark}
    It is clear that, given two matrices $R=R^\top\succeq0$ and $B$, the set
    \begin{equation*}
        \Omega(R,B) = \left\{ c\in\mathbb{R} \mid cR \succeq \tfrac{1}{2}BB^\top \right\}
    \end{equation*}
    is either empty or a closed infinite left-bounded interval of the form $[\overline c,\infty)$, where $\overline{c}\geq 0$.
    Note that in the cases we are interested in we actually have $\overline{c}>0$ since $\overline c=0$ would imply $B=0$.
    Since the error bound in Corollary \ref{cor:bt} holds for any $c\in\Omega$, we are particularly interested in the optimal value $\overline{c}$, or equivalently $\overline{\alpha}=\frac{1}{2\overline{c}}$, for which the error bound is minimal.
    
    One way to find the optimal $\overline{\alpha}$ is offered by Lemma \ref{lem:cRB_cond}.
    Since $R=R^\top\succeq 0$, there exists an invertible matrix $P\in\mathbb{R}^{n\times n}$, such that
    \begin{equation*}
        P^\top RP = \bmat{I_k & 0 \\ 0 & 0}, \quad\text{with}\quad
        P = U\bmat{\Lambda^{-\frac{1}{2}} & 0 \\ 0 & I_{n-k}}
    \end{equation*}
    where $U,\Lambda$ are as in the proof of Lemma \ref{lem:cRB_cond}.
    If there exists a constant $c>0$ satisfying the condition, then because of the same Lemma we have
    \begin{equation*}
        P^\top BB^\top P = \bmat{ \Lambda^{-\frac{1}{2}}S\Lambda^{-\frac{1}{2}} & 0 \\ 0 & 0 }
        = \bmat{ \widehat{S} & 0 \\ 0 & 0 },
    \end{equation*}
    where $S=S^\top\geq0$ is again as in the proof of Lemma \ref{lem:cRB_cond}.
    Let now $\widehat{S}=V^\top\widehat{\Lambda}V$ be the spectral decomposition of $\widehat{S}$, where $V^\top V=I_n$ and $\widehat{\Lambda}\succeq 0$ is diagonal.
    Then it is clear that
    \begin{equation*}
        R \succeq \alpha BB^\top \quad\iff\quad
        P^\top RP \succeq \alpha P^\top BB^\top P \quad\iff\quad
        I_k \succeq \alpha \widehat{S} \quad\iff\quad
        I_k \succeq \alpha \widehat{\Lambda},
    \end{equation*}
    therefore the optimal $\overline{\alpha}$ is the reciprocal of the largest eigenvalue of $P^\top BB^\top P$.
    
    The optimal value $\overline{\alpha}$ has in particular a nice interpretation. Due to Proposition \ref{prop:cRB_cond}: it represents the \emph{radius of structure-preservation} of the pH system \eqref{eq:pH} under output feedback, for a fixed $Q$.
    A possible strategy to decrease the error bound is then to consider a different pH representation, given by a different matrix $X$ (see \eqref{eq:pHX}), such that the radius of structure-preservation is larger.
    Changing $Q$ unfortunately will in general also change the values $\pi_i$, therefore it is not always clear whether maximizing the radius of structure-preservation is a good solution.
\end{remark}

When there is no $c>0$ such that the condition $cR\succeq\frac{1}{2}BB^\top$ is satisfied, we might want to look for a different error bound.
Similarly as in \eqref{eq:popov_aux}, for every $Y=Y^\top \in \mathbb R^{n\times n}$, we may decompose the Popov function as 
\begin{align}
    \Phi(s)=\begin{bmatrix} C(sI_n-A)^{-1} & I_m \end{bmatrix} \underbrace{\begin{bmatrix} -AY - YA^\top & B-YC^\top \\
 B^\top - CY & 0 \end{bmatrix}}_{:=\widehat{W}(Y)} \begin{bmatrix} (-sI_n-A^\top)^{-1}C^\top \\ I_m \end{bmatrix}.
\end{align}
In particular, if $Y$ is such that $\widehat{W}(Y)\succeq 0,$ a Cholesky decomposition yields
\begin{align*}
    \Phi(s)&=\begin{bmatrix} C(sI_n-A)^{-1} & I_m \end{bmatrix} \begin{bmatrix} L_Y^\top \\0  \end{bmatrix} \begin{bmatrix} L_Y & 0 \end{bmatrix}
    \begin{bmatrix} (-sI_n-A^\top)^{-1}C^\top \\ I_m \end{bmatrix} \\
    &= C(sI_n-A)^{-1}L_Y^\top L_Y(-sI_n-A^\top)^{-1}C^\top=V(s)V(-s)^\top,
\end{align*}
where $V(\cdot)=C(\cdot I_n-A)^{-1}L_Y^{\top}$ denotes a spectral factor of the Popov function $\Phi$. For a pH system \eqref{eq:pH}, we immediately have the solution $Y=Q^{-1}$ with the property:
\begin{align*}
    -AY-YA^\top = -AQ^{-1}-Q^{-1}A^\top = -(J-R)-(J^\top-R)=2R.
\end{align*}
Hence, with the Cholesky decomposition $2R=L_R^\top L_R$, we obtain the spectral factor $V(\cdot) = C(\cdot I_n-A)^{-1} L_R^\top $. With regard to balancing of the solutions to \eqref{eq:lyap}, observe that this corresponds to classical balanced truncation of a system described by $V(\cdot)= C(\cdot I_n-A)^{-1} L_R^\top$. As a consequence, we have the following result.

\begin{corollary}\label{cor:spec_bt}
    Let $(J,R,Q,B)$ with asymptotically stable $A=(J-R)Q$ and minimal $(A,R,C)$ define a port-Hamiltonian system of the form \eqref{eq:pH} which is balanced w.r.t.~the solutions of \eqref{eq:lyap}.
    Furthermore, let $(\widehat{A}_r,\widehat{B}_r,\widehat{C}_r)$ be a reduced-order model obtained by truncation of $(A,B,C)$.
    Besides, the balanced Gramians are denoted by $\Mcalo=\Lcalc=\Pi=\mathrm{diag}(\Pi_1,\Pi_2),$ where $\Pi_2=\mathrm{diag}(\pi_{r+1},\dots,\pi_n).$ 
    If $V(\cdot)=C(\cdot I_n-A)^{-1}L_R^\top$ and $\widehat{V}_r(\cdot)=\widehat{C}_r (\cdot I_r - \widehat{A}_r)^{-1} \widehat{L}_{\widehat{R}_r}^\top$ are spectral factors of the associated Popov function, then 
    \begin{align}\label{eq:Hinf_err_bnd_V}
        \|V-\widehat{V}_r\|_{\mathcal{H}_\infty} \le 2 \sum_{i=r+1}^n  \pi_i.
    \end{align}
    Moreover, the reduced-order model $(\widehat{A}_r,\widehat{B}_r,\widehat{C}_r)$ is port-Hamiltonian.
\end{corollary}

\begin{remark}\label{rem:sr-balancing}
While the previous result does not require any condition for $B$, it is based on minimality of the triple $(A,R,C).$ This is however a less restrictive assumption than the one from Corollary \ref{cor:bt}. Indeed, if $(A,B,C)$ is minimal and if there exists $c>0$ as in Corollary \ref{cor:bt}, then with Lemma~\ref{lem:cRB_cond} it follows that $(A,R,C)$ is also minimal. Moreover, it is well-known that minimality can always be ensured by a  \emph{low rank square root} balancing transformation, see, e.g., \cite[Theorem 2.2]{TomP87}, \cite[Section 4.1]{Pen06} and \cite{BenQ99}.  
\end{remark}

\begin{remark}
 The bound \eqref{eq:Hinf_err_bnd_V} is structurally similar to the well-known bounded real balanced truncation error bound for the stable minimum phase spectral factors, see \cite{OpdJ88}.
\end{remark}

\begin{remark}\label{rem:compare_spec_fac}
For the error bound \eqref{eq:Hinf_err_bnd_V}, it is crucial that $L_R^\top $ and $\widehat{L}_{\widehat{R}_r}^\top$ have the same number of columns since otherwise a comparison in the $\mathcal{H}_\infty$-norm does not make sense. A naive computation of $\widehat{L}_{\widehat{R}_r}^\top$ would limit the number of columns by the reduced system dimension and, hence, would generally be smaller than the number of columns of $L_R^\top$. On the other hand, from \eqref{eq:pH_red} it makes sense to define $\widehat{L}_{\widehat{R}_r}^\top$ via $    \widehat{L}_{\widehat{R}_r}^\top = W^\top L_R^\top$
which automatically ensures identical number of columns.
\end{remark}

Following the discussion in Section \ref{subsec:choosing_Q}, we may again ask whether there exists a particularly good pH representation such that the previous error bounds are minimized. 
With this in mind, it is obvious to minimize the eigenvalues $\lambda_i(\Mcalo \Lcalc) = \lambda_i(\Mcalo Q^{-1})$. 
If $Y=Q^{-1}$ is such that $\widehat{W}(Y)\succeq 0$, it then follows that $X=Y^{-1}=Q$ is a solution to \eqref{eq:kyp-lmi} and vice versa.
Consequently, if $X$ is a maximal solution to \eqref{eq:kyp-lmi} and the system is replaced by the alternative pH formulation \eqref{eq:pHX}, then the error bound will be minimal. 
This will be further illustrated in the numerical examples.

Similarly as in Remark~\ref{rem:mixedBalancing}, we note that the approach outlined in the last paragraph can be viewed as a mixed balancing procedure, where the solution of the observability Lyapunov equation in \eqref{eq:lyap} is balanced with the minimal solution of the dual KYP-LMI \eqref{eq:dualKYPLMI}.
Consequently, this approach coincides with the balancing procedure presented in \cite{UnnVE07}.
In contrast to \cite{UnnVE07}, we are not only able to show that the reduced-order models are port-Hamiltonian and, thus, passive, but also to derive two computable a priori error bounds, which hold under some additional assumptions on the dissipation matrix $R$, cf.~Corollaries~\ref{cor:bt} and \ref{cor:spec_bt}.

\section{Numerical examples} 

In this section, we provide two numerical examples that naturally lead to port-Hamiltonian systems. 
Let us emphasize that the focus of this section is on illustrating the theory rather than demonstrating that the new methods outperform existing ones or than applying the methods to large-scale systems. 
Thus, we only consider systems of moderate state space dimension and focus on the error bounds and on demonstrating the influence of the different choices for the Hamiltonian on these bounds.

All simulations were generated on an Intel i5-9400F @ 4.1 GHz x 6,
64 GB RAM,  \matlab \;version R2019b. 

With regard to the implementation of the individual methods, the following remarks are in order:
\begin{itemize}
    \item For obtaining the stabilizing solutions to Riccati equations, we relied on the \matlab\; built-in routine \texttt{icare}. Similarly, for the computations of the $\mathcal{H}_{\infty}$-errors, we used the \texttt{Control System Toolbox}.
    \item The extremal solutions to the KYP-LMI \eqref{eq:kyp-lmi} were computed by a regularization (see \cite[Theorem 2]{Wil72a}) approach based on an artificial feedthrough term $D+D^\top=10^{-12}I_m$ which allowed to replace the LMI
    \begin{align*}
        \begin{bmatrix} -A^\top X - XA & C^\top -XB \\ C-B^\top X & D+D^\top \end{bmatrix} \succeq 0   \hspace{2cm}
    \intertext{by the Riccati equation associated with the Schur complement, i.e.,}
        A^\top X + XA + (C^\top -XB) (D+D^\top)^{-1} (C-B^\top X) = 0.
    \end{align*}
    \item In order to avoid numerically ill-conditioned operations such as, e.g., the computation of the Schur complement in the effort-constraint reduction method, cf.\@ \eqref{eq:eff_const}, numerically minimal realizations have been computed in a preprocessing step.
        For this, we utilized the approach from Section \ref{sec:classic_bt} due to its structure and Hamiltonian preserving nature with a truncation threshold $\varepsilon_{\mathrm{trunc}}=10^{-11}$,  cf.~Remark \ref{rem:sr-balancing}.
\end{itemize}

\subsection{A scalable mass-spring-damper system} \label{sec:massSpringDamper}

The first example is a scalable multiple-input and multiple-output mass-spring-damper system that has been introduced in \cite{Gugetal12}. The system consist of two inputs $u_1,u_2$ acting as external forces applied to the first two masses of the system. In view of the port-Hamiltonian framework, the outputs $y_1,y_2$ are the velocities of the masses. We refrain from a more detailed discussion and instead refer to the original presentation in \cite{Gugetal12}. For our numerical simulations, we follow the parameters used in the latter reference and define the masses $m_i$, spring constants $k_i$ and damping constants $c_i$ as $m_i=4,k_i=4,$ and $c_i=1$ for all $i=1,\dots,n$. Here, we use a system consisting of $\ell=500$ masses, resulting in an original system of dimension $n=1000$ defined by the matrices $(J,R,Q,B)$.

In Figure \ref{fig:msd}, we show the results obtained for the method from Theorem \ref{thm:err_bnd_lqg} for three different port-Hamiltonian realizations as discussed in Section \ref{subsec:choosing_Q}. In particular, besides the canonical Hamiltonian function $\mathcal{H}(x)=\frac12 x^\top Qx$ resulting from the modelling, we also include the results for a Hamiltonian corresponding to the extremal solutions $X_{\mathrm{min}}$ and $X_{\mathrm{max}}$ to \eqref{eq:kyp-lmi}. 
Furthermore, we include a comparison to the classical version of LQG balanced truncation. 
The observations are as follows. As seen in Figure \ref{fig:msd_errorDecayOfTransferFunction}, the error bounds for $\left\|\begin{bsmallmatrix} M \\ N \end{bsmallmatrix}-\begin{bsmallmatrix} \widehat{M}_r \\ \widehat{N}_r \end{bsmallmatrix}\right\|_{\mathcal{H}_\infty}$ follow the actual behavior of the error. Moreover, as predicted by the discussion in Section \ref{subsec:choosing_Q}, the error bound clearly depends on the chosen Hamiltonian, with the maximal solution of \eqref{eq:kyp-lmi}  being favorable compared to other choices. The minimal solution of \eqref{eq:kyp-lmi} yields a closed loop error (bound) that is almost stagnating. While error and error bound for LQG balanced truncation is smallest, let us recall that the reduced controllers will generally not be port-Hamiltonian. 

For investigating the effect on the open-loop behavior, we further show in Figure~\ref{fig:msd_errorDecayOfSpectralFactors} the standard $\mathcal{H}_{\infty}$-error for the different methods. Note that neither for the method from Theorem \ref{thm:err_bnd_lqg} nor for LQG balanced truncation, an error bound is available for this open loop behavior. We also include reduced-order models obtained by effort-constraint balancing w.r.t.\@ the standard LQG Riccati equations, since these are guaranteed to be port-Hamiltonian as well. The conclusions are similar to the closed loop behavior, indicating that if the maximal solution to the KYP-LMI \eqref{eq:kyp-lmi} defines the Hamiltonian, the approximability of the systems is maximized. 
 
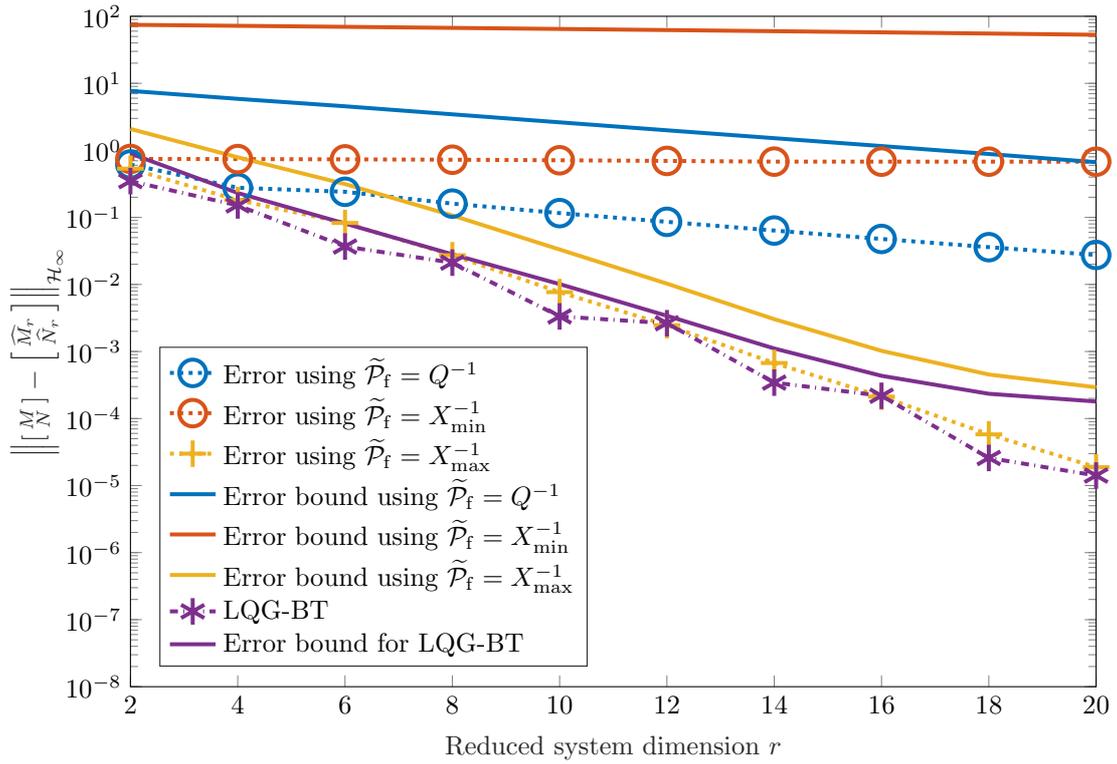
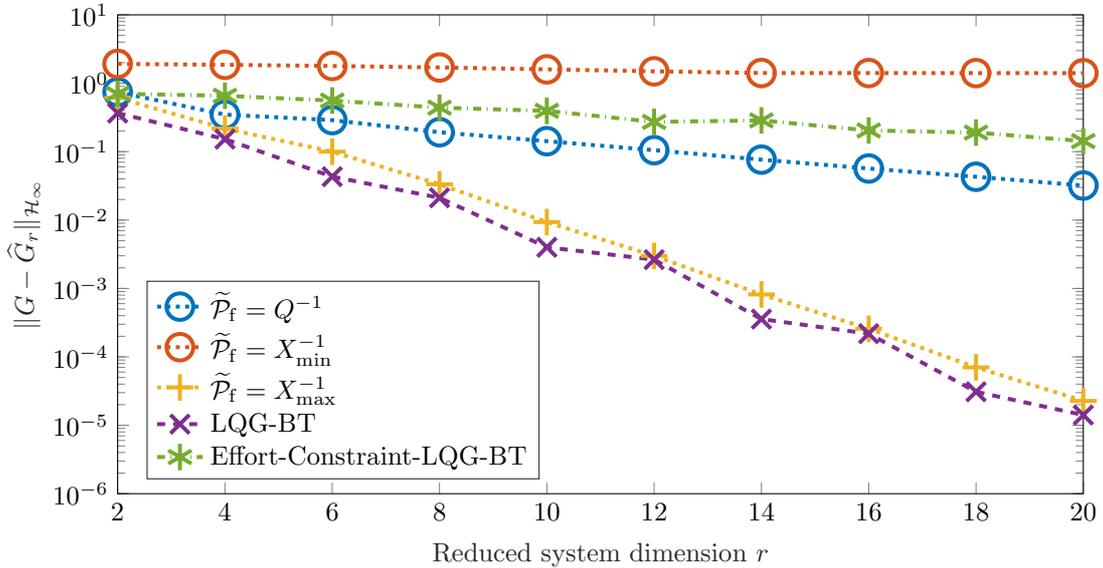
\begin{figure}[tb]
    \centering
    \begin{subfigure}{\textwidth}
        \centering
%
%
\definecolor{mycolor1}{rgb}{0.00000,0.44700,0.74100}%
\definecolor{mycolor2}{rgb}{0.85000,0.32500,0.09800}%
\definecolor{mycolor3}{rgb}{0.92900,0.69400,0.12500}%
\definecolor{mycolor4}{rgb}{0.49400,0.18400,0.55600}%
\begin{tikzpicture}

\begin{axis}[%
width=5in,
height=3.5in,
scale only axis,
xmin=2,
xmax=20,
xlabel style={font=\color{white!15!black}},
xlabel={Reduced system dimension $r$},
ymode=log,
ymin=1e-08,
ymax=100,
yminorticks=true,
ylabel style={font=\color{white!15!black}},
ylabel={$\left\| \begin{bsmallmatrix} M \\ N \end{bsmallmatrix}  - \begin{bsmallmatrix} \widehat{M}_r \\ \widehat{N}_r \end{bsmallmatrix}\right\|_{\mathcal{H}_\infty}$},
axis background/.style={fill=white},
legend style={legend cell align=left, align=left, draw=white!15!black},
legend pos = south west
]
\addplot [color=mycolor1, dotted, line width=1.5pt, mark size=5pt, mark=o, mark options={solid, mycolor1}]
  table[row sep=crcr]{
2	0.616652157870202\\
4	0.277272578220866\\
6	0.241123093052545\\
8	0.160764300631193\\
10	0.115940783061776\\
12	0.0859689116261024\\
14	0.0637436509492213\\
16	0.0477386033993415\\
18	0.0359997061355149\\
20	0.0275555197714684\\
};
\addlegendentry{Error using $\Ptildef=Q^{-1}$}

\addplot [color=mycolor2, dotted, line width=1.5pt, mark size=5pt, mark=o, mark options={solid, mycolor2}]
  table[row sep=crcr]{
2	0.74522932320008\\
4	0.74147357811654\\
6	0.735077637437803\\
8	0.725606570154768\\
10	0.712075632990484\\
12	0.696279087812919\\
14	0.679353439198416\\
16	0.678957918819807\\
18	0.678651300761184\\
20	0.678006482404585\\
};
\addlegendentry{Error using $\Ptildef=X_{\mathrm{min}}^{-1}$}

\addplot [color=mycolor3, dotted, line width=1.5pt, mark size=5pt, mark=+, mark options={solid, mycolor3}]
  table[row sep=crcr]{
2	0.532903708119665\\
4	0.181139568373084\\
6	0.0825848906202101\\
8	0.0272839032547585\\
10	0.00768863826365715\\
12	0.00246865331394027\\
14	0.000673722753640097\\
16	0.000212379172052109\\
18	5.8061918624414e-05\\
20	1.87851492202945e-05\\
};
\addlegendentry{Error using $\Ptildef=X_{\mathrm{max}}^{-1}$}

\addplot [color=mycolor1, line width=1.5pt]
  table[row sep=crcr]{
2	7.71093506765029\\
4	5.86897914960725\\
6	4.5449708395334\\
8	3.45024624405592\\
10	2.62048326995051\\
12	1.99341371775233\\
14	1.52006457167515\\
16	1.16016106416234\\
18	0.883171499753929\\
20	0.66701767783943\\
};
\addlegendentry{Error bound using $\Ptildef=Q^{-1}$}

\addplot [color=mycolor2, line width=1.5pt]
  table[row sep=crcr]{
2	74.6226897754183\\
4	72.1089351268986\\
6	69.6130219796564\\
8	67.1420802388446\\
10	64.7031459934075\\
12	62.3030131122123\\
14	59.9476880769594\\
16	57.6126198523009\\
18	55.284885785315\\
20	52.969437184034\\
};
\addlegendentry{Error bound using $\Ptildef=X_{\mathrm{min}}^{-1}$}

\addplot [color=mycolor3, line width=1.5pt]
  table[row sep=crcr]{
2	2.08624523864027\\
4	0.789367115437814\\
6	0.314572703812744\\
8	0.107028392714165\\
10	0.0332108505604698\\
12	0.0102886322092784\\
14	0.00303592337140188\\
16	0.00102266589472515\\
18	0.00045374624093975\\
20	0.00029261449597759\\
};
\addlegendentry{Error bound using $\Ptildef=X_{\mathrm{max}}^{-1}$}

\addplot [color=mycolor4, dashdotted, line width=1.5pt, mark size=5pt, mark=asterisk, mark options={solid, mycolor4}]
  table[row sep=crcr]{
2	0.34763556134791\\
4	0.151025604304131\\
6	0.0368125920197954\\
8	0.0211461685955804\\
10	0.003335308316417\\
12	0.00263974524691548\\
14	0.000343377933554418\\
16	0.000218395776954716\\
18	2.58144964831721e-05\\
20	1.41170048028696e-05\\
};
\addlegendentry{LQG-BT}

\addplot [color=mycolor4, line width=1.5pt]
  table[row sep=crcr]{
2	0.89683991272191\\
4	0.23138416404063\\
6	0.0810608579082583\\
8	0.0281832764921035\\
10	0.0101282140209546\\
12	0.00340788320141734\\
14	0.00110928517660427\\
16	0.00043379355023259\\
18	0.000234027211272973\\
20	0.000179676900305956\\
};
\addlegendentry{Error bound for LQG-BT}
\end{axis}
\end{tikzpicture}%
        \caption{Decay of the $\mathcal{H}_{\infty}$-error of the coprime factors for the approach from subsection \ref{subsec:modifiedLQG} using three different choices for the Hamiltonian. The solid lines represent the respective error bounds as presented in Theorem \ref{thm:err_bnd_lqg}. As a reference, we also added the error decay for classical LQG balanced truncation and the corresponding error bound.}
        \label{fig:msd_errorDecayOfTransferFunction}
    \end{subfigure}
    \\[0.5cm]
    \begin{subfigure}{\textwidth}
        \centering
%
%
\definecolor{mycolor1}{rgb}{0.00000,0.44700,0.74100}%
\definecolor{mycolor2}{rgb}{0.85000,0.32500,0.09800}%
\definecolor{mycolor3}{rgb}{0.92900,0.69400,0.12500}%
\definecolor{mycolor4}{rgb}{0.49400,0.18400,0.55600}%
\definecolor{mycolor5}{rgb}{0.46600,0.67400,0.18800}%
\begin{tikzpicture}

\begin{axis}[%
width=5in,
height=2.5in,
scale only axis,
xmin=2,
xmax=20,
xlabel style={font=\color{white!15!black}},
xlabel={Reduced system dimension $r$},
ymode=log,
ymin=1e-06,
ymax=10,
yminorticks=true,
ylabel style={font=\color{white!15!black}},
ylabel={$\|G-\widehat{G}_r\|_{\mathcal{H}_\infty}$},
axis background/.style={fill=white},
legend style={legend cell align=left, align=left, draw=white!15!black},
legend pos = south west
]
\addplot [color=mycolor1, dotted, line width=1.5pt, mark size=5pt, mark=o, mark options={solid, mycolor1}]
  table[row sep=crcr]{2	0.744082872810326\\
4	0.34756764933229\\
6	0.288926424116777\\
8	0.192124410323105\\
10	0.142219848663793\\
12	0.105057124458992\\
14	0.0763889779142869\\
16	0.056698566498536\\
18	0.0429236502550877\\
20	0.0319644236037633\\
};
\addlegendentry{$\Ptildef=Q^{-1}$}

\addplot [color=mycolor2, dotted, line width=1.5pt, mark size=5pt, mark=o, mark options={solid, mycolor2}]
  table[row sep=crcr]{2	1.91874543786029\\
4	1.8609578363944\\
6	1.78788668582218\\
8	1.70201746676392\\
10	1.5960444410127\\
12	1.49473364813163\\
14	1.41094706757945\\
16	1.40946933215982\\
18	1.40828582603525\\
20	1.40567294180739\\
};
\addlegendentry{$\Ptildef=X_{\mathrm{min}}^{-1}$}

\addplot [color=mycolor3, dotted, line width=1.5pt, mark size=5pt, mark=+, mark options={solid, mycolor3}]
  table[row sep=crcr]{2	0.611690630086815\\
4	0.217526268909017\\
6	0.0998961024781227\\
8	0.0332709878135346\\
10	0.00935102852517833\\
12	0.00299350312733037\\
14	0.00081684805339807\\
16	0.000257346109241097\\
18	7.03026067168997e-05\\
20	2.27079199574001e-05\\
};
\addlegendentry{$\Ptildef=X_{\mathrm{max}}^{-1}$}

\addplot[color=mycolor4, dashed, line width=1.5pt, mark size=5pt, mark=x, mark options={solid, mycolor4}]
  table[row sep=crcr]{2	0.36614683224583\\
4	0.152329756060591\\
6	0.0431020235579909\\
8	0.0211497151721212\\
10	0.00400021549465604\\
12	0.00263975214523832\\
14	0.000356515456415654\\
16	0.000218395780856599\\
18	3.08902825446482e-05\\
20	1.41246335765469e-05\\
};
\addlegendentry{LQG-BT}

\addplot [color=mycolor5, dashdotted, line width=1.5pt, mark size=5pt, mark=asterisk, mark options={solid, mycolor5}]
  table[row sep=crcr]{2	0.702066095732823\\
4	0.655981005863506\\
6	0.558431807312606\\
8	0.439757675040894\\
10	0.394811880717265\\
12	0.271940954956295\\
14	0.286309898421114\\
16	0.203997088834261\\
18	0.190184047077736\\
20	0.142474354203826\\
};
\addlegendentry{Effort-Constraint-LQG-BT}
\end{axis}
\end{tikzpicture}%
        \caption{Decay of the $\mathcal{H}_\infty$ error of the transfer function for the approach from subsection \ref{subsec:modifiedLQG} using three different choices for the Hamiltonian. As a reference, we also added the error decay for classical LQG balanced truncation and for the effort-constraint reduction based on balancing the standard LQG Riccati equations.}
        \label{fig:msd_errorDecayOfSpectralFactors}
    \end{subfigure}
    \caption{Comparison of closed and open loop errors for a mass-spring-damper system with $n=1000$.}
    \label{fig:msd}
\end{figure}
  
\subsection{Damped wave propagation}
 
In this subsection we consider the set of linear partial differential equations
\begin{alignat}{2}
    a\partial_t p(t,x)  &= -\partial_x q(t,x),\qquad && \text{for all }(t,x)\in (0,\tEnd)\times(0,\ell), \label{eq:massBalance}\\
    b\partial_t q(t,x)  &= -\partial_x p(t,x)-dq(t,x),\qquad && \text{for all }(t,x)\in (0,\tEnd)\times(0,\ell),
    \label{eq:momentumBalance}
\end{alignat}
which describe the damped propagation of pressure waves in a pipeline, see for instance \cite{EggKLMM18}.
The unknowns of this system are the mass flow $q$ and the relative pressure $p$, where the latter one represents the deviation of the absolute pressure $\pAbs$ to a constant reference pressure $\pAbs_0$, i.e., $p\vcentcolon=\pAbs-\pAbs_0$.
Furthermore, $a,b\in\R_{>0}$ are parameters derived from the properties of the fluid and from the geometry of the pipeline, cf{.} \cite{EggKLMM18}.
Besides, the dissipation within the pipeline is modelled by the damping parameter $d\in\R_{>0}$.
Finally, $\tEnd$ denotes the length of the considered time interval and $\ell$ the length of the pipe.

In addition to the partial differential equations \eqref{eq:massBalance} and \eqref{eq:momentumBalance}, we consider homogeneous initial conditions as well as the boundary conditions 
\begin{alignat*}{4}
    p(t,0) &= u_1(t),\qquad &&p(t,\ell) &&= u_2(t),\qquad && \text{for all }t\in (0,\tEnd)
\end{alignat*}
with given functions $u_1$ and $u_2$.

For the semi-discretization in space we use the mixed finite element method as outlined in the appendix of \cite{EggKLMM18}.
To this end, we decompose the computational domain $[0,\ell]$ by an equidistant mesh with $N$ inner grid points and mesh width $h=\frac{\ell}{N+1}$.
Based on this mesh, the relative pressure $p(t,\cdot)$ is approximated by a piecewise constant function and the mass flow $q(t,\cdot)$ by a piecewise linear function.
As a result, we obtain after Galerkin projection the linear system of ordinary differential equations
\begin{equation}
    \label{eq:semidiscreteISO2}
    \begin{bmatrix}
        aM_1  & 0\\
        0     & bM_2
    \end{bmatrix}
    \begin{bmatrix}
        \dot{p}_h(t)\\
        \dot{q}_h(t)
    \end{bmatrix}
    =
    \begin{bmatrix}
        0       & -D\\
        D^\top  & -dM_2
    \end{bmatrix}
    \begin{bmatrix}
        p_h(t)\\
        q_h(t)
    \end{bmatrix}
    +
    \begin{bmatrix}
        0\\
        B_2
    \end{bmatrix}
    u(t).
\end{equation}
Here, the vectors $p_h(t)\in\R^{N+1}$ and $q_h(t)\in\R^{N+2}$ contain the coefficients of $p(t,\cdot)$ and $q(t,\cdot)$ with respect to the corresponding finite element bases.
Furthermore, the coefficient matrices are given by $M_1=hI_{N+1}$,
\begin{alignat*}{2}
    M_2 &= \frac{h}6
    \begin{bmatrix}
        2       & 1         & 0         & \cdots    & 0         & 0\\
        1       & 4         & 1         & \ddots    & \vdots    & \vdots\\
        0       & 1         & 4         & \ddots    & 0         & 0\\
        \vdots  & \ddots    & \ddots    & \ddots    & 1         & 0\\
        0       & \cdots    &  0        & 1         & 4         & 1\\
        0       & \cdots    & 0         & 0         & 1         & 2
    \end{bmatrix}
    \in\R^{(N+2)\times(N+2)},\quad B_2 = 
    \begin{bmatrix}
        1       & 0\\
        0       & 0\\
        \vdots  & \vdots\\
        0       & 0\\
        0      & -1
    \end{bmatrix}
    \in\R^{(N+2)\times2},\quad\text{and}\\ 
    D &= 
    \begin{bmatrix}
        -1      & 1         & 0         & \cdots    & 0\\
        0       & -1        & 1         & \ddots    & \vdots\\
        \vdots  & \ddots    & \ddots    & \ddots    & 0\\
        0       & \cdots    & 0         & -1        & 1
    \end{bmatrix}
    \in\R^{(N+1)\times(N+2)}.
\end{alignat*}
By adding the output equation
\begin{equation}
    \label{eq:outputEq}
    y(t) = 
    \begin{bmatrix}
        0 & B_2^\top
    \end{bmatrix}
    \begin{bmatrix}
        p_h(t)\\
        q_h(t)
    \end{bmatrix}
    ,
\end{equation}
we observe that the system \eqref{eq:semidiscreteISO2} and the output equation \eqref{eq:outputEq} form a port-Hamiltonian system in generalized state space form as in \eqref{eq:pH_gen}.
We summarize this system as
\begin{subequations}
    \label{eq:ISO2pHE}
    \begin{alignat}{1}
        E\dot{z}(t) &= (\tilde{J}-\tilde{R})z(t)+\tilde{B}u(t) \label{eq:ISO2pHE_state}\\
        y(t)        &= \tilde{B}^\top z(t) \label{eq:ISO2pHE_output}
    \end{alignat}
\end{subequations}
with
\begin{equation*}
    z(t) = 
    \begin{bmatrix}
        p_h(t)\\
        q_h(t)
    \end{bmatrix}
    ,\quad E = 
    \begin{bmatrix}
        aM_1    & 0\\
        0       & bM_2
    \end{bmatrix}
    ,\quad \tilde{J} = 
    \begin{bmatrix}
        0       & -D\\
        D^\top  & 0
    \end{bmatrix}
    ,\quad \tilde{R}=
    \begin{bmatrix}
        0 & 0\\
        0 & dM_2
    \end{bmatrix}
    ,\quad \text{and}\quad \tilde{B} = 
    \begin{bmatrix}
        0\\
        B_2
    \end{bmatrix}
    .
\end{equation*}
Finally, we transform \eqref{eq:ISO2pHE} to a pH system of the form \eqref{eq:pH} in the following way.
We first compute the Cholesky decomposition $E=LL^\top$, transform the state via $x=L^\top z$ and multiply \eqref{eq:ISO2pHE_state} from the left by $L^{-1}$.
This leads to the matrices $Q=I_{2N+3}$, $J=L^{-1}\tilde{J}L^{-\top}$, $R=L^{-1}\tilde{R}L^{-\top}$, and $B=L^{-1}\tilde{B}$.
Especially, we obtain a sparse matrix $Q$, but in general dense matrices $J$ and $R$.
However, since the first diagonal block of $\tilde{R}$ is $0$ and the second one a multiple of $M_2$, the matrix $R$ is also block diagonal with the diagonal blocks $0$ and $\frac{d}{b}I_{N+2}$.
Thus, the matrix $R$ is sparse for the example at hand and, furthermore, since the first block of $B$ is $0$ and since the second diagonal block of $R$ is positive definite, there exists $c\in\R_{>0}$ such that $c R\succeq  \frac{1}{2}BB^\top$.
Hence, the assumptions of Corollary \ref{cor:bt} are satisfied and we can employ the corresponding error bound.

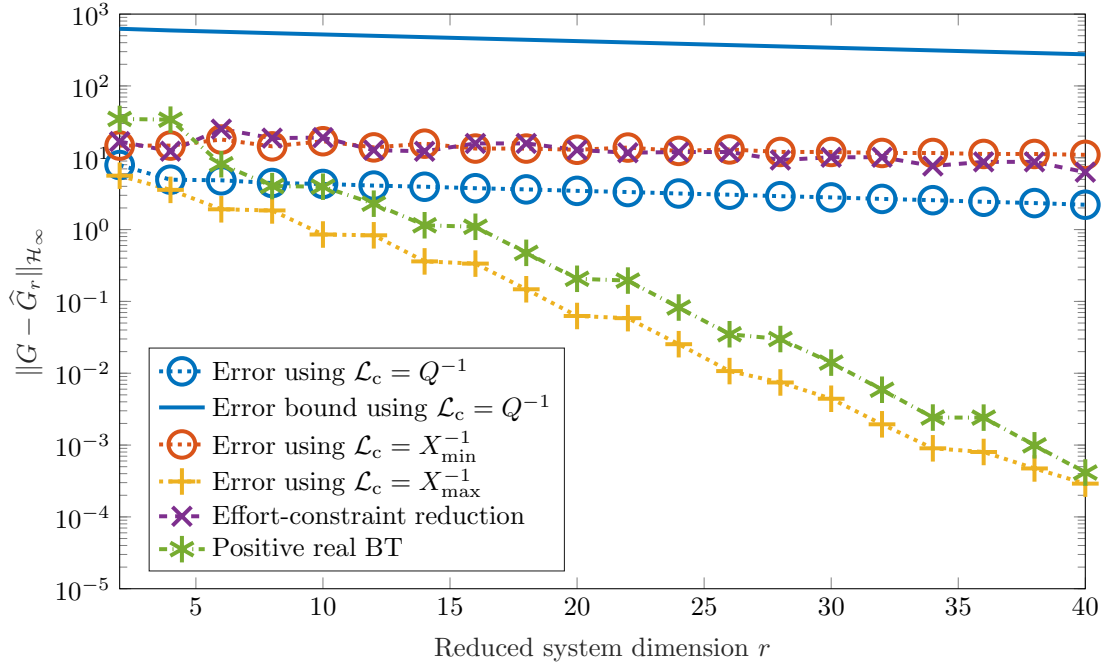
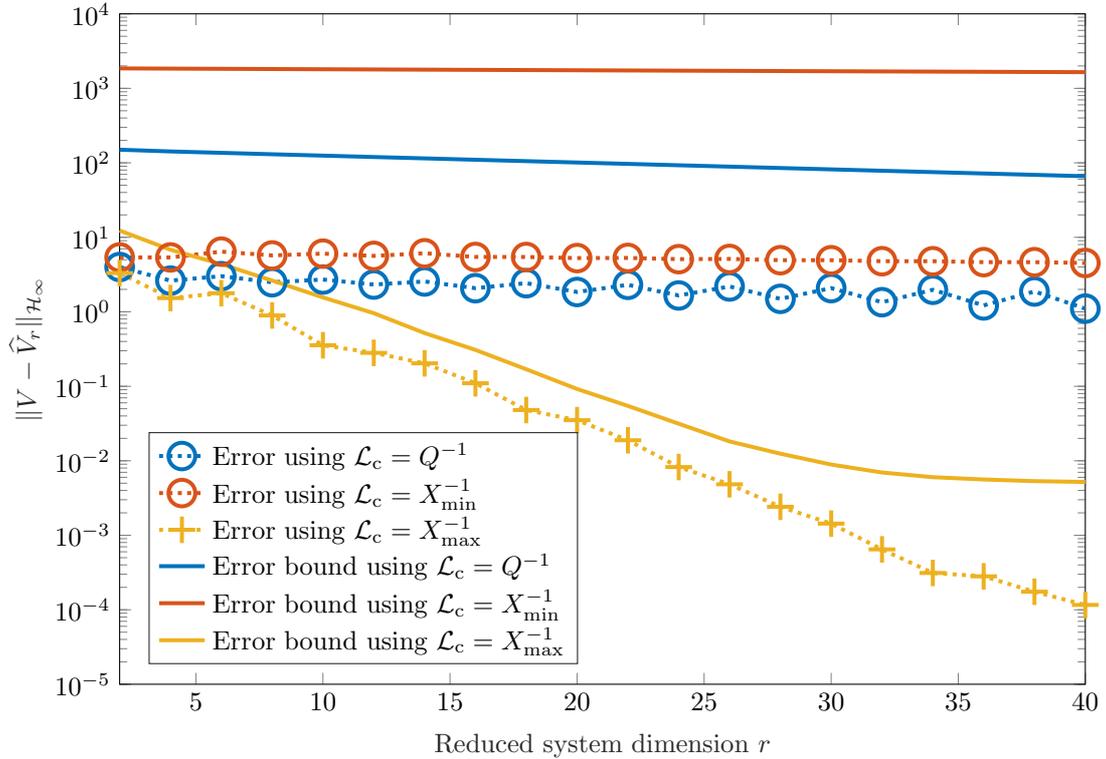
\begin{figure}[tb]
    \centering
    \begin{subfigure}{\textwidth}
        \centering
%
%
\definecolor{mycolor1}{rgb}{0.00000,0.44700,0.74100}%
\definecolor{mycolor2}{rgb}{0.85000,0.32500,0.09800}%
\definecolor{mycolor3}{rgb}{0.92900,0.69400,0.12500}%
\definecolor{mycolor4}{rgb}{0.49400,0.18400,0.55600}%
\definecolor{mycolor5}{rgb}{0.46600,0.67400,0.18800}%
\begin{tikzpicture}

\begin{axis}[%
width=5in,
height=3in,
scale only axis,
xmin=2,
xmax=40,
xlabel style={font=\color{white!15!black}},
xlabel={Reduced system dimension $r$},
ymode=log,
ymin=1e-05,
ymax=1000,
yminorticks=true,
ylabel style={font=\color{white!15!black}},
ylabel={$\|G-\widehat{G}_r\|_{\mathcal{H}_\infty}$},
axis background/.style={fill=white},
legend style={legend cell align=left, align=left, draw=white!15!black},
legend pos = south west
]
\addplot [color=mycolor1, dotted, line width=1.5pt, mark size=5pt, mark=o, mark options={solid, mycolor1}]
  table[row sep=crcr]{%
2	7.91908518510869\\
4	4.95102107967538\\
6	4.85068060880459\\
8	4.45011918476263\\
10	4.31474994845069\\
12	4.09800368966969\\
14	3.94904183687742\\
16	3.76547689839647\\
18	3.62671130674196\\
20	3.45974258573428\\
22	3.3320707691467\\
24	3.17847913299868\\
26	3.05985635977225\\
28	2.91710278046002\\
30	2.80552677776053\\
32	2.67164116438372\\
34	2.56560883179675\\
36	2.43921897118442\\
38	2.33775268346682\\
40	2.21815236228168\\
};
\addlegendentry{Error using $\Lcalc=Q^{-1}$}

\addplot [color=mycolor1, line width=1.5pt]
  table[row sep=crcr]{%
2	624.975321028997\\
4	592.307630014546\\
6	567.337196854698\\
8	542.840452771878\\
10	520.399411443865\\
12	498.612017574509\\
14	477.923983161845\\
16	457.970210278942\\
18	438.962944218373\\
20	420.651267190256\\
22	403.202735153013\\
24	386.398486614299\\
26	370.384387155205\\
28	354.963319374947\\
30	340.266352230282\\
32	326.113608444916\\
34	312.624627836578\\
36	299.634377587531\\
38	287.252457639264\\
40	275.327455730114\\
};
\addlegendentry{Error bound using $\Lcalc=Q^{-1}$}

\addplot [color=mycolor2, dotted, line width=1.5pt, mark size=5pt, mark=o, mark options={solid, mycolor2}]
  table[row sep=crcr]{%
2	14.7595191055664\\
4	14.7596207226521\\
6	17.9557212252386\\
8	14.4147267054613\\
10	16.9834484201322\\
12	13.9411102136183\\
14	15.704117967602\\
16	13.4248887763945\\
18	13.4206942271843\\
20	12.9122239035467\\
22	13.8642372315901\\
24	12.4523519181954\\
26	13.0475416929085\\
28	12.0389932151687\\
30	12.0345939935005\\
32	11.6527921436984\\
34	11.6464358889808\\
36	11.2920773175831\\
38	11.2846856441085\\
40	10.9771499819456\\
};
\addlegendentry{Error using $\Lcalc=X_{\mathrm{min}}^{-1}$}

\addplot [color=mycolor3, dotted, line width=1.5pt, mark size=5pt, mark=+, mark options={solid, mycolor3}]
  table[row sep=crcr]{%
2	5.62892433371896\\
4	3.57235180221314\\
6	1.91599432608539\\
8	1.84660822402967\\
10	0.853753706211504\\
12	0.832592587197532\\
14	0.360641806368109\\
16	0.336011504074012\\
18	0.147547498076009\\
20	0.0626628799944512\\
22	0.0584511113950766\\
24	0.0253739669297533\\
26	0.0107149425363773\\
28	0.00743957605804905\\
30	0.00441467563487866\\
32	0.00194600813466954\\
34	0.000901286233355076\\
36	0.000801146677548693\\
38	0.000472650022337037\\
40	0.000290542093483625\\
};
\addlegendentry{Error using $\Lcalc=X_{\mathrm{max}}^{-1}$}

\addplot [color=mycolor4, dashed, line width=1.5pt, mark size=5pt, mark=x, mark options={solid, mycolor4}]
  table[row sep=crcr]{%
2	16.8137650699798\\
4	12.339951608528\\
6	25.0443605778266\\
8	18.7625903978802\\
10	19.0889135751963\\
12	12.7578429306125\\
14	12.3048287892099\\
16	15.8284790526371\\
18	15.828478939678\\
20	12.8354459978056\\
22	11.7456671407367\\
24	12.0123365395041\\
26	12.01233299042\\
28	9.29000721471662\\
30	10.1599550214504\\
32	10.1599885179486\\
34	7.6851553602578\\
36	8.76358902265702\\
38	8.76445990946556\\
40	6.29668945678695\\
};
\addlegendentry{Effort-constraint reduction}

\addplot [color=mycolor5, dashdotted, line width=1.5pt, mark size=5pt, mark=asterisk, mark options={solid, mycolor5}]
  table[row sep=crcr]{%
2	34.7101018532385\\
4	33.8913700362552\\
6	8.12603803641103\\
8	4.03288330593024\\
10	3.95479373054552\\
12	2.29356973749241\\
14	1.1484217506696\\
16	1.0937563874013\\
18	0.474457550705345\\
20	0.206548881458491\\
22	0.195124909809533\\
24	0.0822211127486992\\
26	0.0349080696030492\\
28	0.0299385130782396\\
30	0.0140594669192126\\
32	0.00589268767717854\\
34	0.00241605698808648\\
36	0.00241115135686612\\
38	0.000991598111011815\\
40	0.000414814580889607\\
};
\addlegendentry{Positive real BT}
\end{axis}
\end{tikzpicture}%
        \caption{Decay of the $\mathcal{H}_\infty$ error of the transfer function for the approach from Section \ref{sec:classic_bt} using three different choices for the Hamiltonian. The solid line represents the error bound from Corollary \ref{cor:bt} for the case of the canonical Hamiltonian. As a reference, we also added the error decay for effort-constrained reduction and for positive real balanced truncation.
        }
        \label{fig:dampedWave_errorDecayOfTransferFunction}
    \end{subfigure}
    \\[0.5cm]
    \begin{subfigure}{\textwidth}
        \centering
%
%
\definecolor{mycolor1}{rgb}{0.00000,0.44700,0.74100}%
\definecolor{mycolor2}{rgb}{0.85000,0.32500,0.09800}%
\definecolor{mycolor3}{rgb}{0.92900,0.69400,0.12500}%
\definecolor{mycolor4}{rgb}{0.49400,0.18400,0.55600}%
\definecolor{mycolor5}{rgb}{0.46600,0.67400,0.18800}%
\begin{tikzpicture}

\begin{axis}[%
width=5in,
height=3.5in,
scale only axis,
xmin=2,
xmax=40,
xlabel style={font=\color{white!15!black}},
xlabel={Reduced system dimension $r$},
ymode=log,
ymin=1e-05,
ymax=10000,
yminorticks=true,
ylabel style={font=\color{white!15!black}},
ylabel={$\|V-\widehat{V}_r\|_{\mathcal{H}_\infty}$},
axis background/.style={fill=white},
legend style={legend cell align=left, align=left, draw=white!15!black},
legend pos = south west
]
\addplot [color=mycolor1, dotted, line width=1.5pt, mark size=5pt, mark=o, mark options={solid, mycolor1}]
  table[row sep=crcr]{%
2	3.99232094811875\\
4	2.62309803139656\\
6	3.00738148918052\\
8	2.49307838330481\\
10	2.70942260768707\\
12	2.32222660728256\\
14	2.55201125053701\\
16	2.07460906750284\\
18	2.42226042976531\\
20	1.84584782647664\\
22	2.30157426171567\\
24	1.65709479331616\\
26	2.18758609270383\\
28	1.49218722726675\\
30	2.07896366966879\\
32	1.34579180717998\\
34	1.97463005883579\\
36	1.21806382380752\\
38	1.87362841265963\\
40	1.10292110487751\\
};
\addlegendentry{Error using $\Lcalc=Q^{-1}$}

\addplot [color=mycolor2, dotted, line width=1.5pt, mark size=5pt, mark=o, mark options={solid, mycolor2}]
  table[row sep=crcr]{%
2	5.37200721277266\\
4	5.37016295115903\\
6	6.42283418047467\\
8	5.71185497486172\\
10	6.03759982969646\\
12	5.63241605824627\\
14	6.10727886760709\\
16	5.44970843563723\\
18	5.44602052342544\\
20	5.26185890657201\\
22	5.28700494539482\\
24	5.09151587634857\\
26	5.12609511065614\\
28	4.93080003337362\\
30	4.92678072591439\\
32	4.77929077868273\\
34	4.77457878040154\\
36	4.64156840775233\\
38	4.63732045270892\\
40	4.52243444994373\\
};
\addlegendentry{Error using $\Lcalc=X_{\mathrm{min}}^{-1}$}

\addplot [color=mycolor3, dotted, line width=1.5pt, mark size=5pt, mark=+, mark options={solid, mycolor3}]
  table[row sep=crcr]{%
2	3.30984994713346\\
4	1.52857250858461\\
6	1.77667619084854\\
8	0.893639360148377\\
10	0.355316036405685\\
12	0.281095947514207\\
14	0.20325769436431\\
16	0.109815308322963\\
18	0.047976621290696\\
20	0.0350790130719933\\
22	0.0188641112089043\\
24	0.00824296905768576\\
26	0.00483351497977412\\
28	0.00241233133533029\\
30	0.00143303273635619\\
32	0.00064655002684412\\
34	0.000311813518922537\\
36	0.000281147244273887\\
38	0.00017468041610004\\
40	0.000116075383709229\\
};
\addlegendentry{Error using $\Lcalc=X_{\mathrm{max}}^{-1}$}

\addplot [color=mycolor1, line width=1.5pt]
  table[row sep=crcr]{%
2	150.019838588515\\
4	142.178246179748\\
6	136.18431293778\\
8	130.304084599807\\
10	124.917309659247\\
12	119.687437052176\\
14	114.721456030411\\
16	109.931728042118\\
18	105.369200706286\\
20	100.973643410487\\
22	96.7852764914158\\
24	92.7515641694179\\
26	88.9075201965484\\
28	85.2058282714213\\
30	81.6779503463328\\
32	78.280708460414\\
34	75.0427970973626\\
36	71.9246015782532\\
38	68.9524304067636\\
40	66.0899383988664\\
};
\addlegendentry{Error bound using $\Lcalc=Q^{-1}$}

\addplot [color=mycolor2, line width=1.5pt]
  table[row sep=crcr]{%
2	1852.32248196414\\
4	1840.01969615587\\
6	1827.98118046361\\
8	1815.94369384528\\
10	1804.26562190726\\
12	1792.58973845584\\
14	1781.30343289244\\
16	1770.01963420853\\
18	1759.11535770135\\
20	1748.21405399671\\
22	1737.66318057904\\
24	1727.1166685789\\
26	1716.88607760096\\
28	1706.66017332043\\
30	1696.71737210924\\
32	1686.77848788016\\
34	1677.09291901844\\
36	1667.41158739941\\
38	1657.95688017607\\
40	1648.50680033432\\
};
\addlegendentry{Error bound using $\Lcalc=X_{\mathrm{min}}^{-1}$}

\addplot [color=mycolor3, line width=1.5pt]
  table[row sep=crcr]{%
2	12.330285745989\\
4	6.81266486934648\\
6	4.36135980427692\\
8	2.630780419518\\
10	1.55978197910598\\
12	0.958083867633112\\
14	0.518016021341949\\
16	0.307783497774376\\
18	0.169288194008117\\
20	0.092017109836423\\
22	0.0545580611061145\\
24	0.0314508303834002\\
26	0.0181949996723522\\
28	0.0124543896650309\\
30	0.00890402526844439\\
32	0.00696101999266514\\
34	0.00602142578124864\\
36	0.0056106524260298\\
38	0.00533080049544021\\
40	0.00520018426095097\\
};
\addlegendentry{Error bound using $\Lcalc=X_{\mathrm{max}}^{-1}$}

\end{axis}
\end{tikzpicture}%
        \caption{Decay of the $\mathcal{H}_\infty$ error of the spectral factors for the approach from Section \ref{sec:classic_bt} using three different choices for the Hamiltonian. The solid lines represent the respective error bounds as presented in Corollary \ref{cor:spec_bt}.}
        \label{fig:dampedWave_errorDecayOfSpectralFactors}
    \end{subfigure}
    \caption{Damped wave equation -- Error decay for the balancing-based model reduction approach introduced in Section \ref{sec:classic_bt}.}
    \label{fig:dampedWave_errorDecay}
\end{figure} 

For the numerical experiments we set $a=1$, $b=1$, $d=50$, and $\ell=1$.
Furthermore, the number of internal grid points is chosen as $N=500$ which results in a state space dimension of $1003$.
The resulting full-order model is used for illustrating the findings of Section \ref{sec:classic_bt}.
In particular, we investigate the influence of choosing different Hamiltonians, as mentioned in the paragraph after Remark \ref{rem:compare_spec_fac}, and we compare the errors with the proposed error bounds.
In particular, in the case of the canonical Hamiltonian function $\frac12x^\top Qx$, the constant $c$ from Corollary \ref{cor:bt} can be calculated by computing the largest eigenvalue of $\frac12BB^\top$ and dividing it by $\frac{d}{b}$.
Here, we exploited again that the second diagonal block of $R$ is a multiple of the identity matrix.
However, this reasoning does not apply when we change the Hamiltonian, and thus $R$, as discussed in subsection \ref{subsec:choosing_Q}.
In this case, the assumptions of Corollary \ref{cor:bt} may be violated and this is in fact what we observe numerically for the example at hand.
Thus, there is no error bound available for $\|G-\widehat{G}_r\|_{\mathcal{H}_\infty}$ in the cases where we replace $Q$ by $X_{\mathrm{max}}$ or $X_{\mathrm{min}}$ in the Hamiltonian.

In Figure \ref{fig:dampedWave_errorDecay} we depict the errors in the transfer function and in the spectral factors of the Popov function over the dimension of the reduced-order models.
When comparing the errors with the respective error bounds, we observe that there is a good agreement between their qualitative behaviors.
Furthermore, we find in Figure \ref{fig:dampedWave_errorDecayOfSpectralFactors} that the choice of the Hamiltonian plays an important role for the error bound as predicted by the theory.
Since the decay of the actual error resembles the decay of the corresponding error bound in this example, the choice of the Hamiltonian can also be observed to have a significant influence on the actual errors.

In addition to the approach introduced in Section \ref{sec:classic_bt}, we also show the corresponding $\mathcal{H}_\infty$ error decays for the effort-constraint reduction method and for positive real balanced truncation (positive real BT) in Figure \ref{fig:dampedWave_errorDecayOfTransferFunction}.
We choose them as reference methods since they are the schemes most related to our new approach in the sense that they are balanced-based model reduction techniques which preserve the port-Hamiltonian structure.
For the example at hand, we observe that the error decay of positive real balanced truncation is similar to the one by our approach when using the Hamiltonian corresponding to $X_{\mathrm{max}}$.
On the other hand, the effort-constraint reduction yields reduced-order models with larger errors which are of the same order of magnitude as the ones obtained by our approach using the Hamiltonian based on $X_{\mathrm{min}}$.

\section{Conclusion}

In this paper we propose new balancing-based methods for controller design and for model order reduction which preserve the structure of linear time-invariant port-Hamiltonian systems without algebraic constraints.
To this end, we first derive a modified LQG balancing approach which ensures that the resulting LQG controller is port-Hamiltonian, by properly choosing the weighting matrices in the algebraic Riccati equations.
Based on this balancing, we introduce a corresponding model reduction method to obtain a low-dimensional port-Hamiltonian controller.
In this context, we also present an a priori error bound in the gap metric, which is structurally the same as the one for standard LQG balanced truncation.
Moreover, we show that the error bound can be improved by replacing the canonical Hamiltonian by one which is based on the maximal solution of the associated KYP linear matrix inequality.

In addition to the modified LQG balancing approach, we also derive a modification of standard balanced truncation in order to ensure preservation of the port-Hamiltonian structure.
For this method, we derive an error bound w.r.t.\@ the spectral factors of the associated Popov function.
For the case that a certain condition on the dissipation port and the input port of the full-order model is satisfied, we also present an a priori bound for the $\mathcal{H}_\infty$ error of the transfer function.
Furthermore, we give an intuitive interpretation of the mentioned condition by showing its relation to the existence of a structure-preserving output feedback.
Finally, the new methods and the corresponding error bounds are illustrated by means of two numerical examples: A mass-spring-damper system and a semi-discretized damped wave equation.

Based on the findings of this paper, an interesting next step is the generalization to port-Hamiltonian differential-algebraic equation systems.
This would allow to also construct low-dimensional port-Hamiltonian controllers for systems with algebraic constraints.
Furthermore, we mainly focus on the theory in this paper, whereas an efficient and robust numerical implementation is not addressed.
These practical considerations are certainly an interesting research direction for the future in order to explore the applicability and competitiveness of the new approaches.

\subsection*{Acknowledgements}

We would like to thank Volker Mehrmann for comments on a previous version of this manuscript. We thank the Deutsche Forschungsgemeinschaft for their support within the project B03 in the Sonderforschungsbereich/Transregio 154 ``Mathematical Modelling, Simulation and Optimization using the Example of Gas Networks". 
Finally, we express our gratitude to the two anonymous reviewers for their comments and suggestions which helped to improve and clarify this manuscript.

\appendix

\section{$Q$-conjugated LQG (reduced) control design}\label{apdxA}

Our choice of the weighting matrices $\tilde{\mathcal{R}}$, $\mathcal{R}_{\mathrm{f}}$, $\tilde{\mathcal{Q}}$ and $\mathcal{Q}_{\mathrm{f}}$ leading to \eqref{eq:lqg_care_ph} and \eqref{eq:lqg_fare_ph} is heavily inspired by \cite{Wuetal18} where the authors propose a different set of matrices to enforce the resulting controller to be port-Hamiltonian. 
In this appendix, we review this method and provide two examples which show that the controller obtained with this approach may generally not be realized as a port-Hamiltonian system. Concerning the reduced-order model obtained by the associated (modified) LQG balanced truncation method, we provide a connection to the new approach from Section \ref{sec:pH-LQG}. 
In particular, one implication of our discussion is that the reduced models from Algorithm \ref{alg:new} and \ref{alg:Wu} share the same error bound.

Let us begin by recalling the precise result concerning port-Hamiltonian control design in Theorem \ref{thm:Wu} and the corresponding model/controller reduction method in Algorithm \ref{alg:Wu}.
 
\begin{theorem}\cite[Theorem 7]{Wuetal18}\label{thm:Wu}
Denote the LQG Gramians $\Pf$, solution of the filter Riccati equation \eqref{eq:lqg_fare} and $\Pc$, solution of the control Riccati equation \eqref{eq:lqg_care}. Consider the LQG problem with the following relation between the covariance matrix $\Rf$ and the weighting matrix $\tilde{\mathcal R}=\Rf$ and with the following relation between the covariance matrix $\Qf$ and the weighting matrix $\tilde{\mathcal Q}$:
\begin{align}\label{eq:Wu_weight_rel}
\Qf = Q^{-1}(2QJ^\top \Pc + 2\Pc JQ + \tilde{\mathcal Q} )Q^{-1}.
\end{align}
In this case the LQG Gramians satisfy the following relation: 
\begin{align}\label{eq:Wu_Gram_rel}
\Pc Q^{-1}  = Q \Pf .
\end{align}
Furthermore, assuming that the port-Hamiltonian system is stable, the control Riccati equation \eqref{eq:lqg_care} and the filter Riccati equation \eqref{eq:lqg_fare} admit a unique solution, the LQG controller is passive and the closed loop system can be written as the feedback interconnection of the port-Hamiltonian system \eqref{eq:pH} with the port-Hamiltonian realization of the LQG controller.
\end{theorem}

\begin{algorithm}[H]
    \caption{$Q$-conjugated LQG reduced controller design (\cite{Wuetal18})} 
    \label{alg:Wu}
    \begin{algorithmic}[1]
      \REQUIRE $J, R,Q,\tilde{\mathcal Q}\in \mathbb R^{n\times n}, B\in \mathbb R^{n\times m}, \tilde{\mathcal R}\in \mathbb R^{m\times m}$
      \ENSURE Reduced-order LQG controller $(\Acr,\Bcr,\Ccr)$ \\
      \STATE Choose $\Rf=\tilde{\mathcal R}$ and $\Qf$ as in \eqref{eq:Wu_weight_rel} and compute $\Pc$ and $\Pf=Q^{-1}\Pc Q^{-1}$ solving \eqref{eq:lqg_care} and \eqref{eq:lqg_fare}.\!\!\!
      \STATE Compute  $T\in \mathbb R^{n\times n}$ defined by $T\Pf T^\top = T^{-\top} \Pc T^{-1} = \mathrm{diag}(\sigma_1,\dots,\sigma_n)$.
      \STATE Balance the system $\Jb=TJT^\top$, $\Rb = TRT^\top$, $\Qb =T^{-\top } Q T^{-1} $, and $\Bb=TB$.
      \STATE Proceed to the reduction by using the effort-constraint method accordingly to \eqref{eq:eff_const}. 
      \STATE Compute  $(\Acr,\Bcr,\Ccr)$ for $(A_r,B_r,C_r)$ based on \eqref{eq:LQG_controllers} and Theorem~\ref{thm:Wu}.
      \end{algorithmic}
\end{algorithm} 
 
 The main idea in \cite{Wuetal18} is to exploit relation \eqref{eq:Wu_Gram_rel} such that the controller \eqref{eq:controller} with $(\Ac,\Bc,\Cc)$ as in \eqref{eq:LQG_controllers} is characterized by the same quadratic Hamiltonian function $\mathcal{H}(x)=\frac{1}{2}x^\top Qx$ that defines the original system \eqref{eq:pH}. In more detail, note that 
 \begin{align*}
     \Ac &= A-B\tilde{\mathcal R}^{-1}B^\top \Pc-\Pf C^\top \Rf^{-1} C \\
     &=A-B\tilde{\mathcal R}^{-1}B^\top Q\Pf Q-\Pf QB\tilde{\mathcal R}^{-1} B^\top Q \\
     &=(J-\underbrace{(R+B\tilde{\mathcal R}^{-1}B^\top Q\Pf+\Pf QB\tilde{\mathcal R}^{-1} B^\top)}_{=:\Rc} )Q , \\[1ex]
     \Cc &= \tilde{\mathcal R}^{-1}B^\top  \Pc = \Rf^{-1} B^\top Q \Pf Q = \Bc^\top Q.
 \end{align*}
 In \cite{Wuetal18}, the authors claim that $\Rc$ is symmetric positive definite, rendering the controller port-Hamiltonian. The argument provided is that an LQG controller yields a stable closed loop system, implying $\Rc=\Rc^\top \succ 0$. Let us emphasize that while the LQG controller produces a stable closed loop system, the controller itself does not generally have to be stable, see \cite{Joh79}. Additionally, consider the following example 
 \begin{align*}
     A=\begin{bmatrix} 1 & 2 \\ -2 & -2 \end{bmatrix} = \begin{bmatrix} 0 & 2 \\ -2 & 0 \end{bmatrix} - \begin{bmatrix} -1 & 0 \\ 0 & 2 \end{bmatrix}
 \end{align*}
which shows that asymptotic stability of $A=J-R$ does not (automatically) yield a positive definite $R$ but might require a change of the Hamiltonian matrix $Q=Q^\top \succ 0$. In the following we provide two examples showing that the statement of the above theorem is generally false, even for the particular choice $\tilde{\mathcal Q}=C^\top C$.

\begin{example}[$\Ac$ unstable for general $\tilde{\mathcal Q}$]\label{ex:unstable}
Consider a port-Hamiltonian system \eqref{eq:pH} where 
\begin{align*}
    J=\begin{bmatrix} 0 & 0 \\ 0 & 0 \end{bmatrix},\ \ R=\begin{bmatrix} 1 & 0 \\ 0 & 1 \end{bmatrix}, \ \ Q=\begin{bmatrix} 2 & 1 \\ 1 & 1 \end{bmatrix}, \ \ B=\begin{bmatrix} 1 & 0 \\ 2 & 1 \end{bmatrix}.
\end{align*}
Let the weighting matrices for the control and filter Riccati equation \eqref{eq:lqg_care} and \eqref{eq:lqg_fare} be as follows 
\begin{align*}
    \tilde{\mathcal R} = \begin{bmatrix} 1 & 0 \\ 0 & 1 \end{bmatrix}=\Rf,\ \ \tilde{\mathcal Q} = \begin{bmatrix} 5 & 4 \\ 4 & 7 \end{bmatrix}.
\end{align*}
Note that  
\begin{align*}
    A^\top + A - BB^\top + \tilde{\mathcal Q} = \begin{bmatrix} -4 & -2 \\ -2 & -2 \end{bmatrix} - \begin{bmatrix} 1 & 2 \\ 2 & 5 \end{bmatrix} + \begin{bmatrix} 5 & 4 \\ 4 & 7 \end{bmatrix} = \begin{bmatrix} 0 & 0 \\ 0 & 0 \end{bmatrix}.
\end{align*}
This implies that $\Pc=I_2$ leading to
\begin{align*}
    \Qf = Q^{-1}(2QJ^\top \Pc + 2\Pc JQ + \tilde{\mathcal Q} )Q^{-1}=\begin{bmatrix} 4 & -7 \\ -7 & 17 \end{bmatrix}.
\end{align*}
One easily verifies that $\Pf=Q^{-1} \Pc Q^{-1}=\left[\begin{smallmatrix} 2 & -3 \\ -3 & 5 \end{smallmatrix}\right]$ as well as 
\begin{align*}
    A \begin{bmatrix} 2 & -3 \\ -3 & 5 \end{bmatrix} + \begin{bmatrix} 2 & -3 \\ -3 & 5 \end{bmatrix} A^\top - \begin{bmatrix} 2 & -3 \\ -3 & 5 \end{bmatrix} Q B B^\top Q \begin{bmatrix} 2 & -3 \\ -3 & 5 \end{bmatrix} + \Qf = 0.
\end{align*}
For the resulting controller, we obtain 
\begin{align*}
    \Ac = A-BB^\top \Pc - \Pf QBB^\top Q = \begin{bmatrix} 2 & 1 \\ -17 & -17 \end{bmatrix}
\end{align*}
which is unstable since $\lambda_1(A_c)\approx 1.0586$ and $\lambda_2(A_c)\approx -16.0586$, contradicting the assertion from Theorem \ref{thm:Wu}.
\end{example}

\begin{example}[$(\Ac,\Bc,\Cc)$ not pH for $\tilde{\mathcal Q}=C^\top C$]\label{ex:not_ph}
  Since we obtained the example by numerical tests, below we only present the first five relevant digits of the matrices. In particular, consider a pH system \eqref{eq:pH} where 
  \begin{align*}
    J&\approx \begin{bmatrix} 0 & 4.1002 & 0.5925 \\ -4.1002 & 0 & -1.9806 \\ -0.5925 & 1.9806 & 0   \end{bmatrix},\ \ R\approx \begin{bmatrix} 1.2267 &-1.6531 & 0.2866 \\ -1.6531 & 3.7295 & -0.9714 \\
    0.2866 & -0.9714 & 0.2996 \end{bmatrix}, \\[2ex] 
    Q &= I_3, \ \  B=C^\top \approx \begin{bmatrix} 1.3703 & 0.8682 & 2.1628 \\ -1.2120 & -0.8492 & -1.9551 \\ 0.1859 & -0.2009  &-0.1078 \end{bmatrix}.
\end{align*}
Choosing $\tilde{\mathcal R}=I_3$ and $\tilde{\mathcal Q}=C^\top C$, the previous system is already \emph{balanced} in the sense that 
\begin{align*}
    \Pc=\Pf \approx \begin{bmatrix}  0.8462 & 0 & 0 \\ 0 & 0.5565 & 0 \\ 0 & 0 & 0.1416 \end{bmatrix}.
\end{align*}
For the controller $(\Ac,\Bc,\Cc)$ it now holds that 
\begin{align*}
    \Ac &= A-BB^\top \Pc - \Pf C^\top C = A-BB^\top \Pc -\Pc BB^\top \approx \begin{bmatrix}
    -13.5962 &  15.0479 &   0.4569 \\
    6.8475 & -10.4210 &  -1.1181 \\
   -0.7281  &  2.8430 &  -0.3241
    \end{bmatrix}, \\[2ex]
    \Bc &= \Pf C^\top \Rf^{-1} = \Pc B \tilde{\mathcal R}^{-1}=\Cc^\top \approx \begin{bmatrix} 
    1.1595 &  0.7346 &  1.8301 \\
   -0.6744 & -0.4725 & -1.0879 \\
    0.0263 & -0.0284 & -0.0153
    \end{bmatrix}.
\end{align*}
As mentioned before, port-Hamiltonian systems can be characterized via the KYP-LMI \eqref{eq:kyp-lmi}. Since here we are interested in $(\Ac,\Bc,\Cc)$ being port-Hamiltonian, we have to find $X=X^\top \succeq0 $ such that 
\begin{align*}
\left[
\begin{array}{cc}
-X\,\Ac - \Ac^{{\top}}X & \Cc^{{\top}} - X\,\Bc \\
\Cc- \Bc^{{\top}}X & 0
\end{array}
\right]\succeq 0.
\end{align*}
Since $\Bc=\Cc^\top \in \mathbb R^{3\times 3}$ and $\Bc$ is regular, it has to hold that $X=I_3.$ However, this implies
\begin{align*}
-\Ac-\Ac^\top = 
\begin{bmatrix}    
    27.1922 & -21.895 & 0.2716 \\
    -21.895 & 20.8416 & -1.7253 \\
    0.2716  & -1.7253 & 0.6481
    \end{bmatrix} \not \succeq 0 
\end{align*}
since the smallest eigenvalue of $-\Ac-\Ac^\top$ is $\lambda_{\mathrm{min}}(-\Ac-\Ac^\top)\approx -0.0445<0.$
\end{example} 
  
Interestingly, similar to the result from Theorem \ref{thm:rom_ph}, balancing a pH system according to Algorithm \ref{alg:Wu} leads to a diagonal form of $\Qb$ in \eqref{eq:part}.
 
\begin{lemma}\label{lem:eff_const_is_trunc}
	Let a minimal port-Hamiltonian system \eqref{eq:pH} be given by $(J,R,Q,B)$ and let $\Rf=\tilde{\mathcal R}=I_m$, $\tilde{\mathcal Q}=C^\top C$ and $\Qf$ be given as in \eqref{eq:Wu_weight_rel}.
    If $\Ab=(\Jb-\Rb)\Qb$, $\Bb$, and $\Cb=\Bb^\top \Qb$ are the system matrices of the corresponding system which is balanced w.r.t.\@ $\Pf$ and $\Pc$ as in \eqref{eq:lqg_care} and \eqref{eq:lqg_fare}, then $\Qb=I_n.$ 
In particular, it holds that $Q_r=\QbUpperLeft$ and, consequently, $(A_r,B_r,C_r)$ is obtained by truncation of $(\Ab,\Bb,\Cb)$. 
\end{lemma}

\begin{proof}
 For a system balanced w.r.t.\@ \eqref{eq:lqg_care} and \eqref{eq:lqg_fare}, with weighting matrices as in Theorem \ref{thm:Wu}, we have that 
 \begin{align*}
     \Upsilon=\mathrm{diag}(\upsilon_1,\dots,\upsilon_n)=\Phatf=\Qb^{-1} \Phatc \Qb^{-1}=\Qb^{-1}\Upsilon \Qb^{-1}.
 \end{align*}
 Since we assumed $(\Ab,\Bb,\Cb)$ to be a minimal realization, the pair $(\Ab,\Cb)$ is observable implying that the solution $\Pc$ to the classical control Riccati equation is positive definite. Hence, we conclude that $\upsilon_i>0,i=1,\dots,n$ as well as 
 \begin{align*}
 I_n=\Upsilon^{-\frac{1}{2}}(\Qb^{-1} \Upsilon \Qb^{-1}) \Upsilon^{-\frac{1}{2}}=(\Upsilon^{-\frac{1}{2}}\Qb^{-1} \Upsilon^{\frac{1}{2}}) (\Upsilon^{\frac{1}{2}} \Qb^{-1} \Upsilon^{-\frac{1}{2}}) = (\Upsilon^{-\frac{1}{2}}\Qb^{-1} \Upsilon^{\frac{1}{2}}) (\Upsilon^{-\frac{1}{2}}\Qb^{-1} \Upsilon^{\frac{1}{2}})^\top.
 \end{align*}
 This shows that $\Upsilon^{-\frac{1}{2}}\Qb^{-1} \Upsilon^{\frac{1}{2}}$ is orthogonal. Since $\Qb^{-1}=(\Qb^{-1})^\top \succ 0$, for its eigenvalues $\lambda_j$ we obtain that 
 \begin{align*}
 \lambda_j(\Upsilon^{-\frac{1}{2}} \Qb^{-1} \Upsilon^{\frac{1}{2}})=\lambda_j (\Qb^{-1})\in \mathbb R_{>0} .
 \end{align*}
 Using the orthogonality of $\Upsilon^{-\frac{1}{2}}\Qb^{-1} \Upsilon^{\frac{1}{2}}$, we conclude that $\lambda_i(\Upsilon^{-\frac{1}{2}}\Qb^{-1} \Upsilon^{\frac{1}{2}})=1$ for $i=1,\dots,n$. In other words, $\Upsilon^{-\frac{1}{2}}\Qb^{-1} \Upsilon^{\frac{1}{2}}=I_n$ which also implies that $\Qb=I_n$. The remaining assertions immediately follow from the definition of $Q_r$.
 \end{proof}
 
 Let us emphasize that Lemma \ref{lem:eff_const_is_trunc} implies that balanced truncation w.r.t.\@ the Gramians from \cite{Wuetal18} automatically preserves the pH structure of $(A,B,C)$ within the reduced-order model $(A_r,B_r,C_r).$ This is a consequence of the approach implicitly utilizing the effort-constraint reduction framework. In view of Example \ref{ex:unstable} and Example \ref{ex:not_ph}, a reduced LQG controller $(\Acr,\Bcr,\Ccr)$ will however generally be not pH. 
 
With the intention of showing that the reduced-order models produced by Algorithm \ref{alg:new} and \ref{alg:Wu} are state space equivalent, let us note that the balancing matrix $T$ in Algorithm \ref{alg:Wu} can be specified explicitly as
\begin{align*}
T=\Upsilon^{-\frac{1}{2}} Z^\top \Lc, \text{ where } \ \Lc^\top \Lc=\Pc, \quad \Lf^\top \Lf = \Pf, \quad \underbrace{\begin{bmatrix}U_1 & U_2 \end{bmatrix}}_{=U}\underbrace{\begin{bmatrix}\Upsilon_1 & 0 \\ 0 & \Upsilon_2 \end{bmatrix}}_{=\Upsilon} \underbrace{\begin{bmatrix} Z_1^\top \\ Z_2 ^\top \end{bmatrix}}_{=Z^\top} = \Lf\Lc^\top .
\end{align*}
Moreover, we also have that $T^{-1}=\Lf^\top U\Upsilon^{-\frac{1}{2}}.$ From Lemma \ref{lem:eff_const_is_trunc}, we already know that $\Qb=I_n$ and, hence, the reduced model from Algorithm \ref{alg:Wu} is obtained from a regular Petrov--Galerkin projection $\mathbb{P}=VW^\top$ of the form
\begin{equation}\label{eq:rom_Wu_proj}
 \begin{alignedat}{3}
    A_r &= W^\top AV, \quad& B_r   &= W^\top B,  \quad& C_r   &= CV, \\
  J_{r} &= W^\top JW, \quad& R_{r} &= W^\top RW, \quad& Q_{r} &= I_r, 
 \end{alignedat}
\end{equation}
with $V=T^{-1}\Vb=\Lf^\top U_1\Upsilon_1^{-\frac{1}{2}}$ and $W^\top=\Wb^\top T=\Upsilon_1^{-\frac{1}{2}}Z_1^\top \Lc$, as in \autoref{sec:preliminaries}.

The next result shows that the reduced model $(A_r,B_r,C_r)$ is state space equivalent to a reduced model that is obtained by truncation of a system balanced w.r.t.\@ the Gramians $\Phatc$ and $\Phatf$, i.e., a model which is computed by Algorithm \ref{alg:new}.
\begin{lemma}
    Let a minimal port-Hamiltonian system \eqref{eq:pH} be given by $(J,R,Q,B)$ and
    let $\Rf=\tilde{\mathcal R}=I_m$, $\tilde{\mathcal Q}=C^\top C$, and $\Qf$ be given as in \eqref{eq:Wu_weight_rel}.
    Furthermore, let $(\widehat{A}_r,\widehat{B}_r,\widehat{C}_r)$ be the corresponding reduced-order model obtained from Algorithm \ref{alg:Wu}.
    Besides, consider a state space transformation of the form $\widehat{A}_r=\widehat{T}A_r \widehat{T}^{-1}$, $\widehat{B}_r=\widehat{T}B_r$, and $\widehat{C}_r=C_r \widehat{T}^{-1}$, where $\widehat{T}=\Upsilon_1^{\frac{1}{4}}$, that produces the equivalent port-Hamiltonian system with $\widehat{J}_r=\widehat{T}J_r\widehat{T}$, $\widehat{R}_r=\widehat{T}R_r\widehat{T}$, and $\widehat{Q}_r=\widehat{T}^{-1}Q_r\widehat{T}^{-1}=\Upsilon_1^{-\frac{1}{2}}$. 
    Then it holds that 
    \begin{align*}
        \widehat{A}_r \Upsilon_1^{\frac{1}{2}} + \Upsilon_1^{\frac{1}{2}} \widehat{A}_r^\top-\Upsilon_1^{\frac{1}{2}}\widehat{C}_r^\top \widehat{C}_r \Upsilon_1^{\frac{1}{2}} + \widehat{B}_r \widehat{B}_r^\top + 2\widehat{R}_r &= 0, \\
        \widehat{A}_r^\top \Upsilon_1^{\frac{1}{2}} + \Upsilon_1^{\frac{1}{2}} \widehat{A}_r-\Upsilon_1^{\frac{1}{2}}\widehat{B}_r\widehat{B}_r^\top \Upsilon_1^{\frac{1}{2}}+\widehat{C}_r^\top \widehat{C}_r=0.
    \end{align*}
\end{lemma}

\begin{proof}
Recalling that $A_r=J_r-R_r$ and $B_r=\BbUpper=Q_r^\top \BbUpper=C_r^\top$, we obtain
\begin{align*}
 &\widehat{A}_r \Upsilon_1^{\frac{1}{2}} + \Upsilon_1^{\frac{1}{2}} \widehat{A}_r^\top-\Upsilon_1^{\frac{1}{2}}\widehat{C}_r^\top \widehat{C}_r \Upsilon_1^{\frac{1}{2}} + \widehat{B}_r \widehat{B}_r^\top \\
 &=\widehat{T} (J_r-R_r)\widehat{T}^{-1}\Upsilon_1^{\frac{1}{2}} + \Upsilon_1^{\frac{1}{2}}\widehat{T}^{-1} (J_r^\top - R_r^\top)\widehat{T}-\Upsilon_1^{\frac{1}{2}}\widehat{T}^{-1}B_r B_r^\top \widehat{T}^{-1} \Upsilon_1^{\frac{1}{2}}+\widehat{T}B_r B_r^\top \widehat{T} \\
 &=\widehat{T} \left((J_r-R_r)+(J_r-R_r)^\top-B_rB_r^\top+B_r B_r^\top \right) \widehat{T} \\
 &= -2 \widehat{T} R_r\widehat{T}=-2\widehat{R}_r.
\end{align*}
Similarly, we obtain 
\begin{align*}
  &\widehat{A}_r^\top \Upsilon_1^{\frac{1}{2}} + \Upsilon_1^{\frac{1}{2}} \widehat{A}_r-\Upsilon_1^{\frac{1}{2}}\widehat{B}_r\widehat{B}_r^\top \Upsilon_1^{\frac{1}{2}}+\widehat{C}_r^\top \widehat{C}_r \\
&=(\widehat{T}(J_r-R_r)\widehat{T}^{-1})^\top \Upsilon_1^{\frac{1}{2}} + \Upsilon_1^{\frac{1}{2}} (\widehat{T}(J_r-R_r)\widehat{T}^{-1}) - \Upsilon_1^{\frac{3}{4}} B_rB_r^\top \Upsilon_1^{\frac{3}{4}}+ \widehat{T}^{-1}B_rB_r^\top \widehat{T}^{-1} \\
&= \widehat{T}^{-1} \left((J_r-R_r)^\top \Upsilon_1 + \Upsilon_1  (J_r-R_r) - \Upsilon_1 B_r B_r^\top \Upsilon_1 + C_r^\top C_r \right) \widehat{T}^{-1} =0,
\end{align*}
where in the last step we have used that $A_r=J_r-R_r, B_r=C_r^\top $ and the fact that $\Upsilon_1$ solves the control Riccati equation for $(A_r,B_r,C_r)$. 
\end{proof}

\bibliographystyle{siam}
\bibliography{references}

\begin{thebibliography}{10}

\bibitem{Ant05a}
{\sc A.~Antoulas}, {\em Approximation of Large-Scale Dynamical Systems},
  Society for Industrial and Applied Mathematics, Philadelphia, PA, USA, jan
  2005.

\bibitem{BeaMV19}
{\sc C.~Beattie, V.~Mehrmann, and P.~Van~Dooren}, {\em Robust
  port-{H}amiltonian representations of passive systems}, Automatica, 100
  (2019), pp.~182--186.

\bibitem{BenQ99}
{\sc P.~Benner and E.~Quintana-Ort\'{i}}, {\em Solving stable generalized
  {L}yapunov equations with the matrix sign function}, Numerical Algorithms, 20
  (1999), pp.~75--100.

\bibitem{BruPM20}
{\sc A.~Brugnoli, D.~Alazard, V.~Pommier-Budinger, and D.~Matignon}, {\em
  Port-{H}amiltonian flexible multibody dynamics}, Multibody System Dynamics,
  (2020).

\bibitem{Cel17}
{\sc E.~Celledoni and E.~H. H{\o}iseth}, {\em Energy-preserving and
  passivity-consistent numerical discretization of port-{H}amiltonian systems},
  2017.
\newblock arXiv 1706.08621.

\bibitem{ChaLVV06}
{\sc Y.~Chahlaoui, D.~Lemonnier, A.~Vandendorpe, and P.~Van~Dooren}, {\em
  Second-order balanced truncation}, Linear Algebra Appl., 415 (2006),
  pp.~373--384.

\bibitem{Cur90}
{\sc R.~F. Curtain}, {\em Robust stabilizability of normalized coprime factors:
  the infinite-dimensional case}, International Journal of Control, 51 (1990),
  pp.~1173--1190.

\bibitem{Cur03}
\leavevmode\vrule height 2pt depth -1.6pt width 23pt, {\em Model reduction for
  control design for distributed parameter systems}, in Research Directions in
  Distributed Parameter Systems, R.~Smith and M.~Demetriou, eds., SIAM,
  Philadelphia, PA, USA, 2003, pp.~95--121.

\bibitem{DamB14}
{\sc T.~Damm and P.~Benner}, {\em {Balanced Truncation for Stochastic Linear
  Systems with Guaranteed Error Bound}}, in Proceedings of the 21st
  International Symposium on Mathematical Theory of Networks and Systems (MTNS
  2014), Groningen, The Netherlands, 2014, pp.~1492--1497.

\bibitem{DuiMSB09}
{\sc V.~Duindam, A.~Macchelli, S.~Stramigioli, and H.~Bruyninckx}, {\em
  Modeling and control of complex physical systems - the port-{H}amiltonian
  approach}, Springer Berlin/Heidelberg, Germany, 2009.

\bibitem{EggKLMM18}
{\sc H.~Egger, T.~Kugler, B.~Liljegren-Sailer, N.~Marheineke, and V.~Mehrmann},
  {\em On structure-preserving model reduction for damped wave propagation in
  transport networks}, SIAM J. Sci. Comput., 40 (2018), pp.~A331--A365.

\bibitem{GolV13}
{\sc G.~H. Golub and C.~F. Van~Loan}, {\em Matrix Computations}, Johns Hopkins
  University Press, Baltimore, USA, fourth~ed., 2013.

\bibitem{GugA04}
{\sc S.~Gugercin and A.~Antoulas}, {\em A survey of model reduction by balanced
  truncation and some new results}, International Journal of Control, 77
  (2004), pp.~748--766.

\bibitem{Gugetal12}
{\sc S.~Gugercin, R.~Polyuga, C.~Beattie, and A.~{van der Schaft}}, {\em
  Structure-preserving tangential interpolation for model reduction of
  port-{H}amiltonian systems}, Automatica, 48 (2012), pp.~1963--1974.

\bibitem{GuiO13}
{\sc C.~Guiver and M.~Opmeer}, {\em Error bounds in the gap metric for
  dissipative balanced approximations}, Linear Algebra Appl., 439 (2013),
  pp.~3659--3698.

\bibitem{Hal94}
{\sc Y.~Halevi}, {\em Stable {LQG} controllers}, IEEE Trans. Automat. Control,
  39 (1994), pp.~2104--2106.

\bibitem{HarJS83}
{\sc P.~Harshavardhana, E.~A. Jonckheere, and L.~M. Silverman}, {\em Stochastic
  balancing and approximation - stability and minimality}, in Proceedings of
  the 22nd IEEE Conference on Decision and Control, San Antonio, TX, USA, 1983,
  pp.~1260--1265.

\bibitem{JacZ12}
{\sc B.~Jacob and H.~Zwart}, {\em Linear Port-{H}amiltonian Systems on
  Infinite-dimensional Spaces}, Springer Basel, Switzerland, 2012.

\bibitem{Joh79}
{\sc C.~Johnson}, {\em State variable design methods may produce unstable
  feedback controllers}, International Journal of Control, 29 (1979),
  pp.~607--619.

\bibitem{KotL19}
{\sc P.~Kotyczka and L.~Lef{\`e}vre}, {\em Discrete-time port-{H}amiltonian
  systems: A definition based on symplectic integration}, Systems \& Control
  Letters, 133 (2019), p.~Article number 104530.

\bibitem{Lil20}
{\sc B.~Liljegren-Sailer}, {\em On port-{H}amiltonian modeling and
  structure-preserving model reduction}, {PhD} thesis, Universit{\"a}t Trier,
  2020.

\bibitem{LozJ88}
{\sc R.~{Lozano-Leal} and S.~{Joshi}}, {\em On the design of the dissipative
  {LQG}-type controllers}, in Proceedings of the 27th IEEE Conference on
  Decision and Control, Austin, TX, USA, 1988, pp.~1645--1646.

\bibitem{McFDG90}
{\sc D.~C. McFarlane and K.~Glover}, {\em Robust controller design using
  normalized coprime factor plant descriptions}, Springer Berlin/Heidelberg,
  Germany, 1990.

\bibitem{MehMS16}
{\sc C.~Mehl, V.~Mehrmann, and P.~Sharma}, {\em Stability radii for linear
  {H}amiltonian systems with dissipation under structure-preserving
  perturbations}, SIAM Journal on Matrix Analysis and Applications, 37 (2016),
  pp.~1625--1654.

\bibitem{Meh91}
{\sc V.~Mehrmann}, {\em The Autonomous Linear Quadratic Control Problem},
  Springer Berlin/Heidelberg, Germany, 1991.

\bibitem{MehM19}
{\sc V.~{Mehrmann} and R.~{Morandin}}, {\em Structure-preserving discretization
  for port-{H}amiltonian descriptor systems}, in 2019 IEEE 58th Conference on
  Decision and Control (CDC), Nice, France, 2019, pp.~6863--6868.

\bibitem{Mey90}
{\sc D.~Meyer}, {\em Fractional balanced reduction: Model reduction via
  fractional representation}, IEEE Trans. Automat. Control, 35 (1990),
  pp.~1341--1345.

\bibitem{MoeRS11}
{\sc J.~M\"{o}ckel, T.~Reis, and T.~Stykel}, {\em Linear-quadratic {G}aussian
  balancing for model reduction of differential-algebraic systems},
  International Journal of Control, 84 (2011), pp.~1627--1643.

\bibitem{Moo81}
{\sc B.~Moore}, {\em Principal component analysis in linear systems:
  Controllability, observability, and model reduction}, IEEE Trans. Automat.
  Control, 26 (1981), pp.~17--32.

\bibitem{MulR76}
{\sc C.~Mullis and R.~Roberts}, {\em Synthesis of minimum roundoff noise fixed
  point digital filters}, IEEE Transactions on Circuits and Systems, 23 (1976),
  pp.~551--562.

\bibitem{OpdJ88}
{\sc P.~Opdenacker and E.~Jonckheere}, {\em A contraction mapping preserving
  balanced reduction scheme and its infinity norm error bounds}, IEEE
  Transactions on Circuits and Systems, 35 (1988), pp.~184--189.

\bibitem{Pen06}
{\sc T.~Penzl}, {\em Algorithms for model reduction of large dynamical
  systems}, Linear Algebra and its Applications, 415 (2006), pp.~322--343.

\bibitem{PolvdS12}
{\sc R.~Polyuga and A.~van~der Schaft}, {\em Effort- and flow-constraint
  reduction methods for structure preserving model reduction of
  port-{H}amiltonian systems}, Systems \& Control Letters, 3 (2012),
  pp.~412--421.

\bibitem{ReiS08}
{\sc T.~Reis and T.~Stykel}, {\em Balanced truncation model reduction of
  second-order systems}, Mathematical and Computer Modelling of Dynamical
  Systems, 14 (2008), pp.~391--406.

\bibitem{SanR04}
{\sc H.~Sandberg and A.~Rantzer}, {\em Balanced truncation of linear
  time-varying systems}, IEEE Trans. Automat. Control, 49 (2004), pp.~217--229.

\bibitem{SefO93}
{\sc J.~A. Sefton and R.~J. Ober}, {\em On the gap metric and coprime factor
  perturbations}, Automatica, 29 (1993), pp.~723--734.

\bibitem{HaMS19}
{\sc A.~Serhani, D.~Matignon, and G.~Haine}, {\em A partitioned finite element
  method for the structure-preserving discretization of damped
  infinite-dimensional port-hamiltonian systems with boundary control}, in
  Geometric Science of Information, Springer, Cham, Switzerland, 2019,
  pp.~549--558.

\bibitem{TomP87}
{\sc M.~Tombs and I.~Postlethwaite}, {\em Truncated balanced realization of a
  stable non-minimal state-space system}, International Journal of Control, 46
  (1987), pp.~1319--1330.

\bibitem{UnnVE07}
{\sc K.~Unneland, P.~Van~Dooren, and O.~Egeland}, {\em A novel scheme for
  positive real balanced truncation}, in American Control Conference (ACC), New
  York, NY, USA, 2007, pp.~947--952.

\bibitem{vdS96}
{\sc A.~van~der Schaft}, {\em $L_2$-Gain and Passivity Techniques in Nonlinear
  Control}, Springer, Cham, Switzerland, 1996.

\bibitem{vdSJ14}
{\sc A.~van~der Schaft and D.~Jeltsema}, {\em Port-{H}amiltonian systems
  theory: An introductory overview}, Foundations and Trends in Systems and
  Control, 1 (2014), pp.~173--378.

\bibitem{vdSM13}
{\sc A.~van~der Schaft and B.~Maschke}, {\em Port-{H}amiltonian systems on
  graphs}, SIAM Journal on Control and Optimization, 51 (2013), pp.~906--937.

\bibitem{Vid84}
{\sc M.~{Vidyasagar}}, {\em The graph metric for unstable plants and robustness
  estimates for feedback stability}, IEEE Trans. Automat. Control, 29 (1984),
  pp.~403--418.

\bibitem{Wil71}
{\sc J.~Willems}, {\em Least squares stationary optimal control and the
  algebraic {R}iccati equation}, IEEE Trans. Automat. Control, 16 (1971),
  pp.~621--634.

\bibitem{Wil72}
\leavevmode\vrule height 2pt depth -1.6pt width 23pt, {\em Dissipative
  dynamical systems part {I}: General theory}, Archive for Rational Mechanics
  and Analysis, 45 (1972), pp.~321--351.

\bibitem{Wil72a}
\leavevmode\vrule height 2pt depth -1.6pt width 23pt, {\em Dissipative
  dynamical systems part {II}: Linear systems with quadratic supply rates},
  Archive for Rational Mechanics and Analysis, 45 (1972), pp.~352--393.

\bibitem{WolLEK10}
{\sc T.~Wolf, B.~Lohmann, R.~Eid, and P.~Kotyczka}, {\em Passivity and
  structure preserving order reduction of linear port-{H}amiltonian systems
  using {K}rylov subspaces}, European Journal of Control, 16 (2010),
  pp.~401--406.

\bibitem{Wu16}
{\sc Y.~Wu}, {\em Passivity preserving balanced reduction for the finite and
  infinite dimensional port Hamiltonian systems}, {PhD} thesis, Universite
  Claude Bernard - Lyon 1, 2016.

\bibitem{WuHLM14}
{\sc Y.~Wu, B.~Hamroun, Y.~Le~Gorrec, and B.~Maschke}, {\em
  Structure-preserving reduction of port {H}amiltonian systems using a modified
  {LQG} method}, in Proceedings of the 33rd Chinese Control Conference,
  Nanjing, China, 2014, pp.~3528--3533.

\bibitem{Wuetal18}
\leavevmode\vrule height 2pt depth -1.6pt width 23pt, {\em Reduced order {LQG}
  control design for port {H}amiltonian systems}, Automatica, 95 (2018),
  pp.~86--92.

\bibitem{ZhoDG96}
{\sc K.~Zhou, J.~Doyle, and K.~Glover}, {\em Robust and optimal control},
  vol.~4, Prentice Hall, aug 1996.

\end{thebibliography}

\end{document}